\documentclass{amsart}
\usepackage[utf8]{inputenc}
\usepackage{amsmath,amssymb,amsthm}
\usepackage{verbatim,enumerate}
\usepackage{graphicx}   
\usepackage{mathrsfs}
\usepackage{bbm}
\usepackage{exscale}

\usepackage{xcolor}

\usepackage{color}

\newtheorem{df}{Definition}
\newtheorem{lemma}{Lemma}
\newtheorem{proposition}{Proposition}

\newtheorem{corollary}{Corollary}

\newtheorem{remark}{Remark}

\newtheorem{theorem}{Theorem}

\theoremstyle{definition}
\newtheorem{example}{Example}
\newtheorem{problem}{Problem}

\newcommand{\dpifactor}{\dfrac{1}{\sqrt{2\pi}}}
\newcommand{\pifactor}{\frac{1}{\sqrt{2\pi}}}

\newcommand{\dsum}{\displaystyle{\sum}}

\newcommand{\R}{
\mathbb{R}}

\newcommand{\Z}{
\mathbb{Z}}
\newcommand{\F}{
\mathcal{F}}
\newcommand{\supp}{
\textnormal{supp}}
\newcommand{\T}{
\mathbb{T}}
\newcommand{\X}{
\mathcal{X}}
\newcommand{\Y}{\mathcal{Y}}
\newcommand{\HH}{\mathcal{H}}

 \newcommand{\ceiling}[1]{
 \lceil#1\rceil}

\newcommand{\psihat}{
\widehat{\psi}}

\newcommand{\bracket}[1]{
\left\langle#1\right\rangle}
\newcommand{\sinc}{\textnormal{sinc}}

\newcommand{\finsum}[3]{
\underset{#1=#2}{\overset{#3}\sum}}

\newcommand{\inflim}[1]{\underset{#1\to\infty}\lim}
\newcommand{\dzsum}[1]{\underset{#1\in\Z}\dsum}

\title[Quasi Shift-Invariant Spaces]{On the Structure and Interpolation Properties of Quasi Shift-invariant Spaces}
\author{Keaton Hamm}
\address{Department of Mathematics, Vanderbilt University, Nashville, TN 37240}
\email{keaton.hamm@vanderbilt.edu}
\author{Jeff Ledford}
\address{Department of Mathematics and Applied Mathematics, Virginia Commonwealth University\\Richmond, VA 23284}
\email{jpledford@vcu.edu}



\begin{document}

\maketitle
\allowdisplaybreaks

\begin{abstract}
The structure of certain types of quasi shift-invariant spaces, which take the form $V(\psi,\X):=\overline{\text{span}}^{L_2}\{\psi(\cdot-x_j):j\in\Z\}$ for an infinite discrete set $\X=(x_j)\subset\R$ is investigated.  Additionally, the relation is explored between pairs $(\psi,\X)$ and $(\phi,\Y)$ such that interpolation of functions in $V(\psi,\X)$ via interpolants in $V(\phi,\Y)$ solely from the samples of the original function is possible and stable.  Some conditions are given for which the sampling problem is stable, and for which recovery of functions from their interpolants from a family of spaces $V(\phi_\alpha,\Y)$ is possible.
\end{abstract}

\section{Introduction}

At the heart of modern signal processing and sampling theory are two fundamental questions: when is sampling a class of functions at a given point-set stable? And if the sampling is stable, how may functions in the given class be reconstructed from their samples?  While these questions may be based around applications and are indeed of import in engineering disciplines, they lead quite quickly to some deep theoretical and structural mathematical problems.  The origin of classical sampling theory lies in the observation that functions in the Paley--Wiener space of bandlimited functions in $L_2(\R)$ whose Fourier transforms are supported in the torus $\T$ can be recovered both in $L_2$ and uniformly on $\R$ by a cardinal sine series:
$$f(x) = \underset{j\in\Z}\dsum f(j)\,\sinc(x-j).$$  This is equivalent to the observation that the exponential system $(e^{-ij\cdot})_{j\in\Z}$ is an orthonormal basis for $L_2(\T)$.

However, much literature in the modern era has been devoted to breaking away from the assumption that functions are bandlimited (which implies, in particular, that the functions are analytic).  Consequently, many interesting function spaces have been discussed -- Wiener amalgam spaces, modulation spaces, spaces of finite rate of innovation, shift-invariant spaces, Sobolev spaces, and Triebel--Lizorkin spaces to name a few.  

In particular, shift-invariant spaces have been popularized in many areas, including functional analysis, harmonic analysis, and approximation theory (in the latter field, they have been well-studied specifically in the setting of radial basis function approximations \cite{DeBoorDeVoreRon,DeBoorDeVoreRon2, DeBoorRon, BuhmannRon,Johnson, Johnson2,K}).  A shift-invariant space has the form $V(\psi):=\overline{\text{span}}\{\psi(\cdot-j):j\in\Z\}$, where the closure is taken in $L_p(\R)$ for some $p\in[1,\infty]$ (see, for example, \cite{Bownik,Jia}).  As a matter of terminology, we call $\psi$ the {\em generator} of the space (note it is sometimes called the window or kernel).  The objects of study in this paper are \emph{quasi shift-invariant spaces}
\[
V(\psi,\X):=\overline{\text{span}}\{\psi(\cdot-x_j):j\in\Z\},
\]
where the closure is taken in $L_2$.  Quasi shift-invariant spaces in this form with $\X$ being infinite have been considered in \cite{Atreas,GS}, and are also implicitly considered in \cite{Sun}.  Recently, stability results for Riesz bases in such spaces were considered in \cite{DV}.  

The purpose of this article is twofold: to determine how the structure of quasi shift-invariant spaces compares to the structure of shift-invariant spaces under the same assumptions on the generators; and secondly, to discuss the recovery of functions in quasi shift-invariant spaces via their interpolants from other quasi shift-invariant spaces.  Motivated by the study of shift-invariant spaces, we consider the following problems.

\begin{problem}\label{PROB0}
Under standard assumptions on the generators $\psi$, and discrete translation sequences $\X\subset\R$, what is the structure of $V(\psi,\X)$?  In particular, is the space a closed subspace of $L_2$?  Do the translates of $\psi$ form a natural basis for the space?
\end{problem}

\begin{problem}\label{PROB}
\begin{enumerate}[a)]
\item Under what conditions on $\psi, \phi, \X,$ and $\Y$ is interpolation of functions in $V(\psi,\X)$ via interpolants in $V(\phi,\Y)$ possible, and uniquely determined?
\item Under what conditions may $f\in V(\psi,\X)$ be recovered from its interpolants in a family of quasi shift-invariant spaces $(V(\phi_\alpha,\Y))_{\alpha\in A}$?
\end{enumerate} 
\end{problem}

Answers to the first problem in the shift-invariant case (i.e. $\X=\Z$) are driven by the fact that $V(\sinc,\Z)=PW_\pi$, which of course is a closed subspace of $L_2(\R)$, and the integer translates of $\sinc$ form an orthonormal basis for $PW_\pi$. The second problem is one of scattered-data interpolation and has been studied by several authors.  For instance, Dyn and Michelli \cite{DM} provide results for finite multivariate interpolation using conditionally positive definite functions, while Jetter and St\"{o}ckler \cite{JetterStockler} examine irregular sampling where one of the spaces is a shift-invariant spline space.  Similar considerations, including reconstruction algorithms may be found in \cite{Aldroubi,AFei}.  More recently, Atreas \cite{Atreas} considers spaces $V(\psi,\X)$ under essentially the same assumptions made below ((A1) and (A2) in Section \ref{sect_psi}) and gives a reconstruction formula in the spirit of the classical Whittaker--Kotel'nikov--Shannon (WKS) sampling theorem; some results in this article may be viewed as approximate reconstruction methods in the same vein.  Radial basis functions (RBFs) have been employed to attack this problem as well, \cite{Narcowich, NarcowichWard}.  The literature on RBF interpolation is vast, and the reader is encouraged to consult \cite{Buhmann_Book} and \cite{wendland} for a more general discussion of this problem.

These works, and the results of \cite{L1}, which solve Problem \ref{PROB} in the special case when $\psi=\sinc$ and the exponentials $(e^{-ix_j\cdot})_{j\in\Z}$ are a Riesz basis for $L_2(\T)$, form the inspiration for our study of this problem. More precisely, Theorems 1 and 2 in \cite{L1} give conditions on families of generators $(\phi_\alpha)$ which allow for recovery of $f\in PW_\pi=V(\sinc,\Z)$ via their interpolants in $V(\phi_\alpha,\Y)$.

Such interpolation schemes have their origin in the work of I. J. Schoenberg on cardinal interpolation via splines \cite{Schoenberg}, which are related to summability methods of the cardinal sine series appearing in the aforementioned WKS sampling theorem.

The rest of the paper is organized as follows. In the following section, we give some preliminary definitions and discuss some of the function spaces to be studied in the sequel.  Section \ref{sect_psi} introduces quasi shift-invariant spaces and gives some indication of their structure as desired by Problem \ref{PROB0}.  Section \ref{sect_interp} begins the study of interpolation between quasi shift-invariant spaces and gives an answer to Problem \ref{PROB}, part a (Theorem \ref{thm_interp_1}), while Section \ref{sect_recovery} provides conditions on families of interpolating generators giving an answer to part b (Theorem \ref{thm_recovery_1} and Corollary \ref{cor_recovery_4}).  The subsequent section provides several types of examples to illustrate the convergence phenomenon described previously, and it turns out that the support of the function $\widehat\psi$ plays an important role in the structure of the examples. Section \ref{sect_cardinal} discusses some extensions to cardinal functions which are classical objects of study in interpolation theory, and we conclude with brief sections on inverse theorems and remarks.

\section{Preliminaries}

Denote by $L_p(\Omega)$ and $\ell_p(I)$ the typical spaces of $p$--integrable functions over a measurable set $\Omega\subset\R$ and $p$--summable sequences indexed by the set $I$, respectively, with their usual norms.  By $L_p$ with no set specified we mean $L_p(\R)$, and likewise $\ell_p:=\ell_p(\Z)$.  For convenience, let $\ell_p':=\ell_p(\Z\setminus\{0\})$.  Additionally, let $C(\Omega)$ be the space of continuous functions on $\Omega$ and $C_0(\R)$ the subset of continuous functions on $\R$ vanishing at infinity.

For functions $f\in L_1$, we will use the following normalization for the Fourier transform:
$$\widehat{f}(\xi):= \dpifactor\int_\R f(x)e^{-i\xi x}dx,$$
whereby the Fourier transform can be extended to a linear isometry on $L_2$.  If $\widehat{f}\in L_1$ and $f$ is continuous, then the inversion formula is $f(x) = \pifactor\int_\R\widehat{f}(\xi)e^{ix\xi}d\xi.$  Plancherel's Identity thus states that $\|\widehat{f}\|_{L_2}=\|f\|_{L_2}.$  We will use $\T$ to denote the torus, which may be identified with the interval $[-\pi,\pi)$.  Unless otherwise specified, $\bracket{\cdot,\cdot}$ is to be taken as the inner product on $L_2(\T)$.

Throughout, $C$ will denote a constant, which may change from line to line depending on context, and subscripts will denote dependence upon a given parameter.  Additionally, the statement $\|\cdot\|_1\asymp\|\cdot\|_2$ will mean that there are constants $c_1$ and $c_2$ such that 
$$c_1\|\cdot\|_1\leq\|\cdot\|_2\leq c_2\|\cdot\|_1.$$

\subsection{Function Spaces}
Throughout the sequel, we will be concerned with many function spaces which arise in applications of harmonic analysis (specifically as different models of the structure of signals one wishes to analyze), but which also enjoy use in other areas of functional analysis.  The first is the classical Paley--Wiener space of {\em bandlimited} functions.  Given $\sigma>0$, let
$$PW_\sigma:=\{f\in L_2: \widehat{f}=0 \textnormal{ a.e. outside of } [-\sigma,\sigma]\},$$
endowed with the norm on $L_2(\R)$.  Since $PW_\sigma$ is isometrically isomorphic to $PW_\gamma$ for any parameters $\sigma$ and $\gamma$ via the map $J_{\sigma\gamma}:PW_\sigma\to PW_\gamma$, $f(x)\mapsto \sqrt{\gamma/\sigma}f(\gamma x/\sigma)$, in the sequel we limit our considerations to the canonical space $PW_\pi$.  Naturally, the Fourier transform is an isometric isomorphism from $PW_\pi$ to $L_2(\T)$.

Of additional utility to our analysis are various {\em Wiener amalgam spaces}.  For $1\leq p,q\leq\infty$, these spaces are defined by
$$W(L_p,\ell_q):=\left\{f : \left(\sum_{j\in\Z}\|f(\cdot+2\pi j)\|_{L_p(\T)}^q\right)^\frac{1}{q}<\infty\right\} ,$$ with the suitable modification when $q$ is infinite. We denote the norm implicit in the definition above via $\|f\|_{W(L_p,\ell_q)}$.  These spaces may also be identified as the $\ell_q$ sum of Banach spaces: $(\oplus_{k\in\Z}L_p(\T+2\pi k))_q$, which is isometrically isomorphic to $\ell_q(L_p(\T))$ via the obvious map.  Note that this readily implies that the amalgam spaces are Banach spaces.  For the special case that is most heavily considered in the sequel, we reduce the notation to $W:=W(L_\infty,\ell_1)$, which is sometimes called Wiener's space.  
Note that these amalgam spaces capture both local and global behavior of the functions simultaneously.  Loosely, functions in $W(L_p,\ell_q)$ are locally in $L_p$ and globally in $\ell_q$.  These spaces, first considered by Wiener, have found great utility in harmonic analysis (see the excellent survey \cite{Heil} and its many references, as well as \cite{Fei,GHO}).  As the sequel will make use of the Fourier transforms of functions in the amalgam spaces, it will be useful to adopt the convention that $\F V$ is the set of Fourier transforms of elements of $V$ provided $V\subset L_2$.

We will not enumerate all of the properties of these amalgam spaces, but let us collect some facts which will be useful later on. 

\begin{proposition}\label{PROPAmalgamBasic}
\begin{enumerate}[(i)]
\item If $1\leq p\leq q\leq\infty$ and $1\leq r\leq\infty$, then $W(L_r,\ell_p)\subset W(L_r,\ell_q)$, and $\|\cdot\|_{W(L_r,\ell_q)}\leq\|\cdot\|_{W(L_r,\ell_p)}$.
\item If $1\leq r\leq\infty$, and $1\leq p\leq q\leq\infty$, then $W(L_q,\ell_r)\subset W(L_p,\ell_r)$, and $\|\cdot\|_{W(L_p,\ell_r)}\leq(2\pi)^{\frac1p-\frac1q}\|\cdot\|_{W(L_q,\ell_r)}$.
\item $W(L_p,\ell_p)=L_p$.
\item $W(L_\infty,\ell_1)\subset L_1\cap L_2.$
\item $\F W(L_\infty,\ell_1)\subset C_0\cap L_2.$
\item If $f,g\in W(L_\infty,\ell_1)$ then $fg\in W(L_\infty,\ell_1)$.
\end{enumerate}
\end{proposition}

\begin{proof}
Note that \textit{(i)} follows from the inclusion $\ell_p\subset\ell_q$, and the fact that the $\ell_q$ norm is subordinate to the $\ell_p$ norm, whilst \textit{(ii)} follows from the facts that $L_q(\T)\subset L_p(\T)$ and $\|\cdot\|_{L_p(\T)}\leq(2\pi)^{\frac1p-\frac1q}\|\cdot\|_{L_q(\T)}.$

Part \textit{(iii)} is evident by the definition of the norm on $W(L_p,\ell_p)$, and \textit{(iv)} follows from combining \textit{(i)}, \textit{(ii)} and \textit{(iii)}.  Finally, \textit{(v)} arises from \textit{(iv)} and the Riemann--Lebesgue Lemma, and \textit{(vi)} follows from the fact that $\ell_1$ is closed under multiplication.
\end{proof}

For a separable Hilbert space $\HH$, let $\mathcal{B}(\HH)$ be the space of bounded linear operators from $\HH$ into itself, and recall that the {\em strong operator topology} (SOT) on $\mathcal{B}(\HH)$ is the topology of pointwise convergence.  That is, $(T_n)\subset\mathcal{B}(\HH)$ converges in SOT to $T\in\mathcal{B}(\HH)$ provided $\|T_nh-Th\|_{\HH}\to0$ for every $h\in\HH$.  When the context is clear, we simply use $\|\cdot\|$ to denote the typical operator norm for elements of $\mathcal{B}(\HH)$.

\subsection{Riesz Bases and Complete Interpolating Sequences}
Two important objects for our subsequent study will be Riesz bases for Hilbert spaces and {\em complete interpolating sequences} for Paley--Wiener spaces.  These objects are defined in rather different ways, but  are nonetheless intimately related.

\begin{df}\label{DEFRieszBasis}
Let $\HH$ be a separable, infinite dimensional Hilbert space.  A family $(h_j)_{j\in\Z}\subset\HH$ is a {\em Riesz basis} for $\HH$ provided it is complete and there exists a constant $C\geq1$ (called the Riesz basis constant) such that for every finite sequence of scalars $(a_j)$,
\begin{equation}\label{EQRB}\frac{1}{C}\|a\|_{\ell_2}\leq\left\|\sum_{j}a_jh_j\right\|_{\HH}\leq C\|a\|_{\ell_2}.\end{equation}
\end{df}
Equivalently, a Riesz basis is a bounded unconditional basis, or is the image of an orthonormal basis under an invertible bounded linear operator, or is an exact, tight frame \cite{Ole,young}.  Moreover, if $(h_j)$ is a Riesz basis for $\HH$, then each $f\in\HH$ admits a unique representation of the form $f=\sum_{j\in\Z}c_jh_j$ with $(c_j)\in\ell_2$.

\begin{df}\label{DEFCIS}
A sequence $\X:=(x_j)_{j\in\Z}\subset\R$ is a {\em complete interpolating sequence} (CIS) for $PW_\pi$ provided for every $(c_j)_{j\in\Z}\in\ell_2$, there exists a unique $f\in PW_\pi$ satisfying
\begin{equation}\label{EQMoment}
f(x_j)=c_j,\quad j\in\Z.
\end{equation}
\end{df}
Finding a solution to \eqref{EQMoment} is called the {\em moment problem} \cite{young}.  It turns out that these two concepts are closely related as the following theorem shows. 
\begin{theorem}[\cite{young}, Theorem 9, p. 143]\label{THMRBCIS}
A sequence $\X$ is a CIS for $PW_\pi$ if and only if $(e^{-ix_j\cdot})_{j\in\Z}$ is a Riesz basis for $L_2(\T)$.
\end{theorem}
Note that Definition \ref{DEFCIS} and Theorem \ref{THMRBCIS} provide a natural bijection between $\ell_2$ and $L_2(\T)$ via the map $(c_j)_{j\in\Z}\mapsto \sum_{j\in\Z}c_je^{-ix_j\cdot}$.  

A necessary condition for $\X\subset\R$ to be a CIS for $PW_\pi$ is for it to be quasi-uniform, i.e. there are $0\leq q\leq Q<\infty$ such that $q\leq|x_{j+1}-x_j|\leq Q$ (assuming that the points are ordered so that $x_j<x_{j+1}$ for all $j$).  A sufficient condition may be found in Kadec's 1/4--theorem \cite{kadec}, which states that if $\sup_{j\in\Z}|x_j-j|<1/4$, then $\X$ is a CIS for $PW_\pi$.  Consequently, CISs are available in abundance, and restricting attention to such sequences is not overly strict. Necessary and sufficient conditions were given by Pavlov \cite{pavlov} in terms of zeros of certain types of entire functions and Muckenhoupt weights.

Next, we define two operations, extension and prolongation, which will be used extensively in the subsequent sections.  First, assume that $\X$ is a CIS for $PW_\pi$, and note that on account of Theorem \ref{THMRBCIS}, if $h\in L_2(\T)$, it admits a unique representation $h=\sum_{j\in\Z}c_je^{-ix_j\cdot}$ in $L_2(\T)$.  Therefore, \eqref{EQRB} implies that the extension  $$E_{\X}(h)(t):=\sum_{j\in\Z}c_je^{-ix_jt},\quad t\in\R$$ is locally square-integrable, and thus is well-defined almost everywhere on $\R$.  Second, for every $k\in\Z$, define the \textit{prolongation operator} $A_{\X}^k:L_2(\T)\to L_2(\T)$ associated with the CIS $\X$ via 
$$A_{\X}^k(h)(t):=E_{\X}(h)(t+2\pi k) = \sum_{j\in\Z}c_je^{-ix_j(t+2\pi k)},\quad t\in\T.$$   Note $A_\X^k$ is not merely translation, as it is viewed as an operator mapping into $L_2(\T)$, so is better viewed as an operation on the coefficients.  An application of \eqref{EQRB} implies that $$\|A_{\X}^k\|\leq C_{\X}^2,\quad k\in\Z,$$
where $C_\X$ is the Riesz basis constant for the exponential system associated with $\X$ as in Definition \ref{DEFRieszBasis}.  The same bound holds for the adjoint, $A_{\X}^{\ast k}$.  Consequently, the families $(A_\X^k), (A_\X^{\ast k})$ are uniformly bounded subsets of $\mathcal{B}(L_2(\T))$.

In the sequel, we will distinguish between prolongation operators for multiple CISs, and if $\mathcal{Y}$ is another CIS, then $E_{\mathcal{Y}}$ and $A_{\mathcal{Y}}^k$ are defined analogously via the expansion of $h$ in terms of the Riesz basis $(e^{-iy_j\cdot})_{j\in\Z}$.

Finally, each Riesz basis has with it an associated dual Riesz basis.  Indeed, if $(\widetilde{e_j})_{j\in\Z}$ are the coordinate functionals associated with the basis $(e^{-ix_j\cdot})_{j\in\Z}$, i.e. the dual elements such that $\bracket{\widetilde{e_j},e^{-ix_k\cdot}}=\delta_{j,k}$, then any $h\in L_2(\T)$ admits a unique representation in each basis as follows:

$$h=\sum_{j\in\Z}\bracket{h,e^{-ix_j\cdot}}\widetilde{e_j} = \sum_{j\in\Z}\bracket{h,\widetilde{e_j}}e^{-ix_j\cdot}.$$
These coordinate functionals are the dual Riesz basis.  It should be noted that the functions $\widetilde{e_j}$ need not be continuous in general, though if $x_j=j$ for every $j$, then evidently $\widetilde{e_j}=e^{-ij\cdot}$.

Our use of the prolongation and extension operators above stem from Lyubarskii and Madych \cite{LMSpline}, who additionally provide a pleasant extension of the classical Poisson Summation Formula to sequences which are CISs for $PW_\pi$ \cite{LMCIS}.

\section{Quasi Shift-Invariant Spaces and their Structure}\label{sect_psi}

Let us now introduce the primary object of the subsequent study: the so-called {\em quasi shift-invariant spaces} (so-named in \cite{GS}).  For the moment, suppose that $\X:=(x_j)_{j\in\Z}\subset\R$ is a discrete sequence, i.e. $\inf_{j\neq k}|x_j-x_k|>0$, and that $\psi\in L_2$.  Then define
$$V(\psi,\X):=\overline{\text{span}}\{\psi(\cdot-x_j):j\in\Z\},$$ where the closure is taken in $L_2(\R)$.  In the special case $\X=\Z$, which is one of the primary motivations of this study, the space $V(\psi):=V(\psi,\Z)$ is called the {\em principal shift-invariant space} associated with the {\em generator} $\psi$ (which is also commonly called the kernel or window function; our use of the term generator follows from \cite{AG}).  For an excellent survey of the weighted sampling problem in shift-invariant spaces, the reader is invited to consult \cite{AG}.  Let us note that in the case that $\{\psi(\cdot-x_j):j\in\Z\}$ is a Riesz basis for $V(\psi,\X)$, then the space can be described in a manner that is easier to handle in practice.  Namely, \begin{equation}\label{Vtilde}V(\psi,\X) = \left\{\sum_{j\in\Z}c_j\psi(\cdot-x_j):(c_j)\in\ell_2\right\},\end{equation} where convergence of the series is taken to be in $L_2$.  Indeed, in much of the literature, the space is defined by \eqref{Vtilde} (for example, \cite{AG,GS}).

In light of the fact that quasi shift-invariant spaces are generalizations of shift-invariant spaces, it is natural to ask how similar their structure is.  Consequently, there are some natural questions that present themselves here, the first of which is: when is $V(\psi,\X)$ a closed subspace of $L_2$? Secondly, when is $\{\psi(\cdot-x_j):j\in\Z\}$ a Riesz basis for $V(\psi,\X)$?  The third question of interest to us (Problem \ref{PROB}) is that of when the nonuniform sampling problem is well-posed, i.e. when is $f$ uniquely and stably determined by its values $(f(y_j))$ for some discrete $\mathcal{Y}\subset\R$?  This of course requires pointwise evaluation to be well-defined in the space $V(\psi,\X)$.  The answers, as will be demonstrated, are similar to the shift-invariant space case.  

\subsection{Regular Generators}

We now turn to some regularity conditions on the generator $\psi$ and the translation set $\X$, and elucidate the structure of the associated quasi shift-invariant spaces. Henceforth we assume that $\psi$ enjoys the following properties:
\begin{enumerate}
    \item[(A1)] $\psi \in L_2(\mathbb{R})$ is positive definite, and 
    \item[(A2)] $\widehat\psi \in W(L_\infty, \ell_1)$ such that
    \[
    C_\psi:= \dfrac{\sum_{j\neq 0}\| \widehat{\psi}(\cdot+2\pi j) \|_{L_\infty(\T)} }{\inf_{\xi\in\T}\vert  \widehat{\psi}(\xi)\vert }<\infty
    \]
\end{enumerate}
For convenience later on, if $\psi$ satisfies (A1) and (A2), then we define
\begin{equation}\label{delta}
\delta_{\psi}:=\inf_{\xi\in\T}\vert  \widehat{\psi}(\xi)\vert,
\end{equation}
which is necessarily positive. 

Note that by Proposition \ref{PROPAmalgamBasic}, (A2) implies that the generator $\psi$ is continuous.  Also, in the case of $V(\psi,\Z)$, properties (A1) and (A2) imply that $\{\psi(\cdot-j):j\in\Z\}$ is a Riesz basis for its closed linear span, and consequently that $V(\psi,\Z)$ is a closed subspace of $L_2$.  Additionally, these assumptions are not overly strict in that they encompass the well-studied sinc kernel, which forms the launching point of our study (and many others, for that matter).  Moreover, the literature on interpolation and approximation of functions by positive definite kernels is extensive, and employs connections with reproducing kernel Hilbert spaces, polynomial splines, and other techniques. See, for example, \cite{NarcowichWard}, and for many references on scattered-data approximation, \cite{wendland}.

For the duration of this work, we will make the assumption that the shift sequence $\X$ is a CIS for $PW_\pi$, or equivalently, via Theorem \ref{THMRBCIS}, that it contains the frequencies of a Riesz basis of exponentials for $L_2(\T)$.  This restriction, while a special case of the general quasi shift-invariant space, comes from the motivation of the interpolation problem found in Theorem 1 of \cite{L1}.   Throughout this section, we will use the equivalent definition in \eqref{Vtilde}, which is justified on account of the following proposition, which is the main structural statement on quasi shift-invariant spaces in this section.

\begin{proposition}\label{lem_psi_1}
If $\X$ is a CIS for $PW_\pi$ and $\psi$ satisfies \emph{(A1)} and \emph{(A2)}, then the following hold:
\begin{enumerate}[(i)]
    \item $\{\psi(\cdot-x_j):j\in\Z\}$ is a Riesz basis for $V(\psi,\X)$, and moreover, 
    \[ \dfrac{\delta_\psi}{C_{\X}^2}\|c\|_{\ell_2}\leq \left\|\sum_{j\in\Z} c_j\psi(\cdot-x_j)\right\|_{L_2(\R)}\leq C_\X^2\|\widehat\psi\|_{W(L_\infty,\ell_1)}\|c\|_{\ell_2};\]
    \item $V(\psi,\X)$ is a closed subspace of $L_2$;
    \item $V(\psi,\X)\subset C_0\cap L_2(\R)$.
    \item For all $f\in V(\psi,\X)$,
    \[
    \| \widehat{f} \|^2_{L_2(\T)}\leq \| f  \|^2_{L_2(\mathbb{R})}\leq \left(1 +C^4_{\X}C^2_{\psi}  \right)\| \widehat{f} \|^2_{L_2(\T)}.
    \]
\end{enumerate}
\end{proposition}

\begin{proof}
Note that \textit{(ii)} follows immediately from \textit{(i)}. Item \textit{(i)} is \cite[Theorem 2.4]{Jetter}; since we use similar periodization arguments often in the sequel, we choose to give the proof here.  Since there is only one CIS of interest, let $A^k=A_{\X}^k$ for ease of notation in the remainder of the proof.  In light of Plancherel's Identity, it suffices to check the Riesz basis inequality in \textit{(i)} with the middle term replaced by the norm of its Fourier transform.  To wit, notice that
\begin{align*}
    \int_\R |\widehat\psi(\xi)|^2\left|\sum_{j\in\Z}c_j e^{-ix_j\xi}\right|^2d\xi & = \sum_{k\in\Z}\int_{\T}|\widehat\psi(\xi+2\pi k)|^2\left|A^k\left(\sum_{j\in\Z}c_je^{-ix_j\xi}\right)\right|^2d\xi\\
    & \leq \|\widehat\psi\|_{W(L_\infty,\ell_1)}^2C_\X^4\|c\|_{\ell_2}^2,\\
\end{align*}
where we have used periodization, the operator bound on $A^k$, the Riesz basis inequality, and the fact that $\|\cdot\|_{\ell_2}\leq\|\cdot\|_{\ell_1}$.  The above inequality gives the right-hand side of the stated Riesz basis inequality for $\{\psi(\cdot-x_j):j\in\Z\}$.  To prove the lower bound, simply note that
\begin{align*}
\int_\R|\widehat\psi(\xi)|^2\left|\sum_{j\in\Z}c_je^{-ix_j\xi}\right|^2d\xi & \geq \int_\T |\widehat\psi(\xi)|^2\left|\sum_{j\in\Z}c_je^{-ix_j\xi}\right|^2d\xi\\
& \geq \dfrac{\delta_\psi^2}{C_\X^4}\|c\|_{\ell_2}^2,\\
\end{align*}
which is the desired lower bound.

In light of the Riemann--Lebesgue Lemma, statement \textit{(iii)} only requires that we check that the Fourier transform of $f\in V(\psi,\X)$ is in $L_1\cap L_2(\R)$.  That $\widehat f$ is square-integrable follows from \textit{(i)}, whereas the calculation to show that $\widehat f\in L_1$ follows by the same method, and so is omitted.  It should be noted that while the above calculation only requires $\widehat\psi\in W(L_\infty,\ell_2)$, showing that $\widehat{f}\in L_1$ requires that $\widehat\psi\in W(L_\infty,\ell_1)$, which is (A2).

The first inequality in \textit{(iv)} is a trivial consequence of Plancherel's Identity.  The second inequality depends upon Plancherel's Identity and periodization. Let $f = \sum c_j\psi(\cdot - x_j)$, then
\begin{align*}
    \int_\mathbb{R} | f(x)|^2 dx &= \int_{\mathbb{R}} |\widehat f(\xi)|^2 d\xi\\
    &=\int_{\T} |\widehat f(\xi)|^2 d\xi + \sum_{k\neq 0} \int_{\T} \vert  \widehat f(\xi+2\pi k)\vert^2d\xi\\
    &=\int_{\T} |\widehat f(\xi)|^2 d\xi + \sum_{k\neq 0} \int_{\T} \left\vert \widehat\psi(\xi+2\pi k)A^k\left( \sum c_je^{ix_j\cdot}  \right)(\xi)   \right\vert^2 d\xi\\
    &\leq \| \widehat f \|^2_{L_2(\T)} + C_\X^4 C_\psi^2 \| \widehat f \|_{L_2(\T)}^2. 
\end{align*}
The inequality follows from elementary estimates, the uniform bound on the prolongation operators $A^k$, multiplying and dividing by $\widehat\psi$, and the fact that $\|\cdot \|_{\ell_2}\leq \| \cdot \|_{\ell_1}$.

\end{proof}

It should be noted that Proposition \ref{lem_psi_1} and its proof imply that $\F V(\psi,\X)\subset \cap L_2$, and that the Fourier inversion formula holds for functions in $V(\psi,\X)$.  Note also, from \textit{(iv)} above, that $V(\psi,\X)$ is isomorphic to $PW_\pi$ if $\X$ is a CIS for $PW_\pi$.

\begin{remark}
It should be noted that, under more relaxed assumptions on $\psi$, its translates need not form a Riesz basis for $V(\psi,\X)$.  Indeed, if $\widehat\psi\in C(\T)$, but $\supp(\widehat\psi)\subsetneq\T$, then $\{\psi(\cdot-x_j):j\in\Z\}$ might fail to be a Riesz basis, or even a frame for its closed linear span (see \cite{Ole} for the definition of a frame). 
\end{remark}

Interestingly, the support of $\widehat\psi$ plays an important role in the structure of these spaces under the assumptions above as the following proposition shows.

\begin{proposition}\label{PROPSuppPsi}
Suppose $\psi$ satisfies \emph{(A1)} and \emph{(A2)}.  Then
\begin{enumerate}[(i)]
    \item  If $\X$ is a CIS for $PW_\pi$ and $\supp(\widehat\psi)=\T$, then $V(\psi,\X)=PW_\pi$.
    \item  If $\supp(\widehat\psi)\supsetneq\T$, then there exist CISs for $PW_\pi$, $\X$ and $\Y$, with $\X\neq\Y$ and $V(\psi,\X)\neq V(\psi,\Y)$.
\end{enumerate}
\end{proposition}

\begin{proof}
Proof of \textit{(i)}: Since the Fourier transform is an isometric isomorphism between $PW_\pi$ and $L_2(\T)$, it suffices to show that $\F V(\psi,\X)=L_2(\T)$.  Since $(e^{-ix_j\cdot})_{j\in\Z}$ is a Riesz basis for $L_2(\T)$, we have that $\F V(\psi,\X) = \{\widehat{\psi}Q:Q \in L_2(\T)\}$, which is $L_2(\T)$ since $\widehat{\psi}(\xi)\geq\delta_\psi>0$ on $\T$.

Proof of \textit{(ii)}: Suppose that $\Y=\Z$ and $\X=\Z\setminus\{1\}\cup\{\sqrt2\},$ and let $f=\psi(\cdot-\sqrt2)\in V(\psi,\X)$.  Then $\widehat{f}=\widehat{\psi}e^{-i\sqrt2\cdot}$.  However, any $g\in V(\psi,\Z)$ satisfies $\widehat{g}=\widehat{\psi}Q$ where $Q$ is a $2\pi$--periodic function.  Consequently, $\widehat{f}$ and $\widehat{g}$ must differ in $L_2(\T+2\pi k)$ for all but a single $k\in\Z$.
\end{proof}

The simple example of item \textit{(ii)} turns out to be an important one and will be used again to discuss the limitations of the interpolation and recovery schemes developed in subsequent sections.  Further discussion of the case when $\supp(\widehat\psi)\supsetneq\T$ is postponed until Section \ref{SECExamples}.

A natural question is how large can the subspaces $V(\psi,\X)$ be?  Can we obtain all of $L_2(\R)$ in this manner?  Such questions have been considered for general $p$ and translation sets $\X$ \cite{AO,FOSZ,Olevskii,OSSZ}.  The answer is essentially that translations can span $L_2$ and $L_p$ for $2< p<\infty$; however, one cannot generate an unconditional basis for all of $L_p(\R)$ in this manner for any $1\leq p\leq \infty$.

\section{Interpolation}\label{sect_interp}

Having elucidated some of the structure of the quasi shift-invariant spaces above, we now turn to the task of interpolation of functions in one such space from another, with the ultimate goal of classifying how well interpolation of functions in a target space $V(\psi,\X)$ with functions in $V(\phi,\Y)$ performs.

Henceforth, we will fix a particular choice of $\psi$ satisfying (A1) and (A2) and a CIS $\mathcal{X}$, and we wish to interpolate data sampled from $V(\psi,\mathcal{X})$ at some sequence $\mathcal{Y}$.  We require that our interpolant have the form
\begin{equation}\label{eq_interp_1}
I_{\phi}^\Y f(x)=\sum_{j\in\mathbb{Z}}a_j\phi(x-y_j),
\end{equation}
(that is, $I_\phi^\Y f\in V(\phi,\Y)$) where $\mathcal{Y}:=(y_j)$ is a CIS for $PW_\pi$ and $(a_j)$ are the interpolating coefficients.  Here, $\phi$ is the {\em interpolation generator}.  

On the way, we first demonstrate that the sampling problem is indeed well-posed as long as one samples at a CIS:

\begin{lemma}\label{lem_psi_2}
Let $\mathcal{X}$ and $\mathcal{Y}$ be CISs for $PW_\pi$, and suppose $\psi$ satisfies \emph{(A1)} and \emph{(A2)}.  If $f\in V(\psi,\mathcal{X}) $ with $f=\sum_{n\in\Z} c_n\psi(\cdot-x_n)$, then $\left( f(y_j)  \right)\in\ell_2$.  Moreover,
\[\|(f(y_j))\|_{\ell_2}\leq \dfrac{1}{\sqrt{2\pi}}C_{\X}^2C_{\Y}^3\left(\delta_\psi^{-1}\| \widehat\psi \|_{L_\infty(\T)}+C_\psi  \right)\| f \|_{L_2(\mathbb{R})}. \] 
\end{lemma}

\begin{proof}
We begin by noting that pointwise evaluation is well-defined by Proposition \ref{lem_psi_1}\textit{(iii)}.  Recalling that $\bracket{\cdot,\cdot}$ is the inner product on $L_2(\T)$, we have
\begin{align*}
f(y_j) & = \pifactor\int_\mathbb{R}\widehat{f}(\xi)e^{i y_j\xi}d\xi \\
& = \pifactor\int_{\mathbb{R}}\widehat\psi(\xi)\left( \sum_{n\in\mathbb{Z}}c_ne^{-ix_n\xi}   \right)e^{iy_j\xi}d\xi\\
& =\pifactor\sum_{k\in\mathbb{Z}}\int_\T \psihat(\xi+2\pi k) A_{\X}^k\left(\sum_{n\in\mathbb{Z}}c_ne^{-ix_n\cdot}   \right)(\xi) \overline{A_{\Y}^k(e^{-iy_j\cdot})(\xi)}d\xi \\
& =\pifactor\left\langle \sum_{k\in\mathbb{Z}} A_{\Y}^{*k}\left[\psihat(\cdot+2\pi k)A_{\X}^k\left(\sum_{n\in\mathbb{Z}}c_ne^{-ix_n\cdot}   \right)\right], e^{-iy_j\cdot} \right\rangle.
\end{align*}
Now we may use the Riesz basis inequality for $(e^{-iy_j\cdot})$, which yields

\[
\|\left( f(y_j) \right) \|_{\ell_2}\leq \dfrac{1}{\sqrt{2\pi}}C_\Y \left\|  \sum_{k\in\mathbb{Z}} A_{\Y}^{*k}\left[\psihat(\cdot+2\pi k)A_{\X}^k\left(\sum_{n\in\mathbb{Z}}c_ne^{-ix_n\cdot}   \right)\right]   \right\|_{L_2(\T)}.
\]
Using the triangle inequality and the bound for $A^{*k}_{\Y}$ and $A^{k}_{\X}$ provides the upper bound
\[
\pifactor C_{\X}^2C_{\Y}^3\| \widehat{\psi} \|_{W} \left\| \sum_{n\in\Z}c_ne^{-ix_n\cdot}  \right\|_{L_2(\T)},
\]
which we may bound above in terms of $\| f \|_{L_2(\R)}$.  Multiplying and dividing by $\widehat\psi$ provides us with the bound
\begin{align*}
\pifactor C_{\X}^2C_{\Y}^3\delta_\psi^{-1}\| \widehat{\psi} \|_{W} \| \widehat{f} \|_{L_2(\T)} &= \pifactor C_{\X}^2C_{\Y}^3\left(\delta_\psi^{-1}\| \widehat\psi \|_{L_\infty(\T)}+C_\psi  \right) \| \widehat{f} \|_{L_2(\T)}\\
&\leq \pifactor C_{\X}^2C_{\Y}^3\left(\delta_\psi^{-1}\| \widehat\psi \|_{L_\infty(\T)}+C_\psi  \right) \| f \|_{L_2(\R)},
\end{align*}
which completes the proof.
\end{proof}

A similar argument allows us to find the following bound in terms of $\| c \|_{\ell_2}$:
\[\|(f(y_j))\|_{\ell_2}\leq \pifactor C_{\X}^3C_{\Y}^3\|\widehat\psi\|_W \| c\|_{\ell_2}.\]

\subsection{Existence of Interpolants}

The following theorem demonstrates that interpolation is well-defined for a large variety of generators.

\begin{theorem}\label{thm_interp_1}
Suppose that $\X$ and $\mathcal{Y}$ are CISs for $PW_\pi$, $\psi$ satisfies (A1) and (A2), and $f\in V(\psi,\X)$.  If $\phi$ satisfies \emph{(A1)} and \emph{(A2)}, then there exists a unique sequence $(a_j)\in\ell_2$ such that the function $I_{\phi}^\Y f$ of \eqref{eq_interp_1} satisfies:
\begin{enumerate}
\item[(i)] $I_{\phi}^\Y f(y_k) = f(y_k)$ for all $k\in\mathbb{Z}$, and
\item[(ii)] $I_{\phi}^\Y f\in C_0\cap L_2(\mathbb{R})$.
\end{enumerate}
\end{theorem}
\begin{proof}
This is essentially a reformulation of Corollary 1 and Proposition 1 in \cite{L1} since (A1) and (A2) imply the conditions found there.  
\end{proof}

We note the following bound on the interpolating coefficients as in Theorem \ref{thm_interp_1}:
\begin{equation}\label{data_bnd}
\| (a_j) \|_{\ell_2}\leq C^4_{\X}\left(\| \widehat\phi \|_{L_\infty(\T)}\delta_{\phi}^{-1}+C^2_{\X}C_\phi \right) \| \left( f(y_j)  \right) \|_{\ell_2}. 
\end{equation}

Note that Theorem \ref{thm_interp_1} implies that the {\em interpolation operator} $I_\phi^\Y$ is a bounded linear operator from $V(\psi,\X)\to C_0\cap L_2(\R)$.  See also \cite{Hamm_Zonotopes, Ledford_Bivariate} for similar interpolation results in higher dimensions.

This interpolation theorem provides us with the following norm equivalences in the quasi shift-invariant space.

\begin{theorem}\label{THMNormEquivalence}
Suppose $\psi$ satisfies \emph{(A1)} and \emph{(A2)}, and that $\X$ is a CIS for $PW_\pi$.  Then if $f=\sum_{j\in\Z}c_j\psi(\cdot-x_j)\in V(\psi,\X)$, we have $$\|f\|_{L_2}\asymp\|(c_j)\|_{\ell_2}\asymp\|(f(x_j))\|_{\ell_2}.$$
\end{theorem}
\begin{proof}
The equivalence $\|f\|_{L_2}\asymp\|(c_j)\|_{\ell_2}$ follows from  Proposition \ref{lem_psi_1}(i).  Meanwhile, the fact that $\|(c_j)\|_{\ell_2}\leq C\|(f(x_j))\|_{\ell_2}$ follows from \eqref{data_bnd} when $\phi=\psi$ and $\Y=\X$.  In particular, this uses the fact that $\widehat\psi\in W(L_\infty,\ell_1)$ implies that the bi-infinite matrix $(\psi(x_j-x_k))_{j,k\in\Z}$ defines an operator in $\mathcal{B}(\ell_2)$ (which in turn relies on the fact that $\X$ is a CIS).  Finally, that $\|(f(x_j))\|_{\ell_2}\leq C\|(c_j)\|_{\ell_2}$ follows from the proof of Lemma \ref{lem_psi_2}.  Putting together these estimates yields the conclusion of the theorem.
\end{proof}

Note that the following stems directly from the above theorem.

\begin{corollary}\label{CORSamplingOperator}
If $\psi$ and $\X$ are as in Theorem \ref{THMNormEquivalence}, then for any $\phi$ satisfying \emph{(A1)} and \emph{(A2)}, the interpolation operator $I_\phi^\X:V(\psi,\X)\to V(\phi,\X)$ is boundedly invertible.  Moreover, the {\em sampling operator} $S_\X:V(\psi,\X)\to\ell_2$ associated with $\X$ defined by $f\mapsto(f(x_j))_{j\in\Z}$ is an isomorphism.
\end{corollary}

Consequently, $V(\psi,\X)$ and $V(\psi,\Y)$ are isomorphic to each other if both $\X$ and $\Y$ are CISs for $PW_\pi$.  However, they are, in general, not the same subspace of $L_2$ as Proposition \ref{PROPSuppPsi}\textit{(ii)} shows. It should also be noted that in the case $\phi=\psi$ but $\X\neq\Y$, interpolation is delicate.  In particular, Baxter and Sivakumar \cite{BS} showed that if $\psi(x)=e^{-|x|^2}$, $\X=\Z$, and $\Y=\Z+\frac12$, then the interpolation operator $I_\psi^\Y:V(\psi,\X)\to V(\psi,\Y)$ is not boundedly invertible.

Gr\"{o}chenig and St\"{o}ckler \cite{GS} and recently together with Romero \cite{GS2} give more general sampling results similar to the second part of Corollary \ref{CORSamplingOperator} when $\psi$ is a {\em totally positive} function, but for quasi-uniform $\X$.  Moreover, their techniques, using Gabor frame analysis, are quite different than the ones used here, and have some rather interesting implications. 

\subsection{Bounds for the Interpolation Operators}

In this subsection, we explore some properties of the interpolation operators $I_\phi^\Y:V(\psi,\X)\to V(\phi,\Y)$ with the two-fold aim of extracting information about general interpolation properties between quasi shift-invariant spaces and setting the stage for the recovery results in the sequel.

In what follows, let $\X$ and $\Y$ be fixed, but arbitrary CISs for $PW_\pi$, and recall that, given a function $\phi$ satisfying (A1) and (A2), each $f\in V(\psi,\X)$ has a unique interpolant $I_{\phi}^\Y f\in V(\phi,\Y)$ via Theorem \ref{thm_interp_1} which satisfies $I_{\phi}^\Y f(y_k)=f(y_k)$, $k\in\Z$.

To move forward, we define some auxiliary operators whose importance will be revealed shortly.  We begin with a simple multiplication operator: for $g\in L_2(\T)$, let
\[
M_\phi g:=\dfrac{\delta_{\phi}}{ \widehat\phi} g,\]
where $\delta_\phi$ is defined as in \eqref{delta}. 
Given a function $\phi$ and an integer $k$, we denote multiplication by $\delta_{\phi}^{-1}\widehat\phi(\cdot+2\pi k)$ by $T_{\phi,k}$; that is, for $g\in L_2(\T)$,
\[
T_{\phi,k} g:= \delta_{\phi}^{-1}\widehat\phi(\cdot+2\pi k) g.
\]
First, let us note that the conditions (A1) and (A2) guarantee that these are well defined, bounded operators on $L_2(\mathbb{T})$.  Indeed we have the transparent bounds
\[ \|M_\phi\|\leq1,\]
and
$ \|T_{\phi,k}\|\leq C_\phi$ if $k\neq0$, with $\|T_{\phi,0}\|\leq \delta_\phi^{-1}\|\widehat\phi\|_{L_\infty(\T)}$,
whence $\|T_{\phi,k}\|$ is bounded independent of $k$.  Of particular note is that (A2) implies that 
\begin{equation}\label{EQTphiSumBound}
\sum_{k\neq0}\|T_{\phi,k}g\|_{L_2(\T)}\leq C_\phi\|g\|_{L_2(\T)}.\end{equation}
Finally, let $B_{\phi},\widetilde{B}_\phi:L_2(\T)\to L_2(\T)$ be defined by
\[
B_{\phi}g:= \sum_{k\neq 0}A_{\Y}^{*k}\left[T_{\phi,k} A_{\X}^k g \right]
\]
and
\[
\widetilde{B}_\phi g:=\sum_{k\neq0}A_{\Y}^{*k}\left[T_{\phi,k}A_{\Y}^kg\right].
\]

Via the same method of calculation found in the previous proofs, one obtains the following bounds for these operators:
\begin{equation}\label{EQBphiBound}
\|B_\phi\|\leq C_\Y^2C_\X^2C_\phi,
\end{equation}
and
\begin{equation}\label{EQTildeBphiBound}
 \|\widetilde{B}_\phi\| \leq C_\Y^4C_\phi.\end{equation}

Note that $\widetilde{B}_\phi$, $T_{\phi,k}$, and $M_\phi$ are positive (i.e. $\bracket{M_\phi g,g}\geq0$ for every $g\in L_2(\T)$); however, $B_\phi$ is not positive in general.

We simplify the notation of the interpolants by naming the nonharmonic Fourier series that arises in their Fourier transform: $$\widehat{I^\Y_{\phi}f}(\xi) = \sum_{j\in\Z}a_je^{-iy_j\xi}\widehat{\phi}(\xi)=:u(\xi)\widehat{\phi}(\xi).$$  The following lemma will be exploited frequently.  

\begin{lemma}\label{lem_recovery_1}
Under the assumptions of Theorem \ref{thm_interp_1}, let $f\in V(\psi,\X)$, and $I_\phi^\Y f$ be its unique interpolant in $V(\phi,\Y)$.  
The following equality holds almost everywhere on $\T$:
\[
(I+B_{\psi}M_{\psi})\widehat{f}=(I+\widetilde{B}_{\phi}M_{\phi})\widehat{I_{\phi}^\Y f},
\]
where $I$ denotes the identity operator on $L_2(\T)$.
\end{lemma}

\begin{proof}
The proof uses periodization and the fact that $\mathcal{Y}$ is a CIS.  Knowing that $f(y_j)=I^\Y_{\phi}f(y_j)$, we simply expand these in terms of the Fourier transform.  First, note that 
\begin{align*}
    \sqrt{2\pi}I^\Y_{\phi}f(y_j)=&\int_{\mathbb{R}}\widehat{\phi}(\xi)u(\xi) e^{iy_j\xi}d\xi\\
    =& \sum_{k\in\mathbb{Z}} \int_\T \widehat{\phi}(\xi+2\pi k) A_{\Y}^ku(\xi) \overline{A_{\Y}^k(e^{-iy_j\cdot})(\xi)}d\xi\\
    =& \sum_{k\in\mathbb{Z}} \int_\T A_{\Y}^{*k}\left[ \widehat{\phi}(\cdot+2\pi k) A_{\Y}^ku\right](\xi) \overline{e^{-iy_j\xi}}d\xi\\
    =& \left\langle \sum_{k\in\mathbb{Z}}A_{\Y}^{*k}\left[ \widehat{\phi}(\cdot+2\pi k) A_{\Y}^ku\right]    ,e^{-iy_j\cdot}\right\rangle\\
    =& \left\langle (I+\widetilde{B}_{\phi}M_{\phi})\widehat{I_{\phi}f},  e^{-iy_j\cdot}   \right\rangle,
\end{align*}
where the inner product is on $L_2(\T)$.

 Similarly, we have 
\begin{align*}
\sqrt{2\pi}f(y_j)=& \left\langle \sum_{k\in\mathbb{Z}} A_{\Y}^{*k}\left[\widehat\psi(\cdot+2\pi k) A_{\X}^k\left(\sum_{n\in\mathbb{Z}}c_ne^{ix_n\cdot}\right)\right], e^{-iy_j\cdot} \right\rangle \\
=& \left\langle \widehat{f}+\sum_{k\neq 0} A_{\Y}^{*k}\left[\widehat\psi(\cdot+2\pi k) A_{\X}^k\left(\sum_{n\in\mathbb{Z}}c_ne^{ix_n\cdot}\right)\right], e^{-iy_j\cdot} \right\rangle\\
=& \left\langle  (I+B_\psi M_\psi)\widehat{f}    , e^{-iy_j\cdot} \right\rangle.
\end{align*}
 
Thus the conclusion of Lemma \ref{lem_psi_2}, the fact that $\mathcal{Y}$ is a CIS, and that the equalities above hold for all $j\in\mathbb{Z}$ completes the proof.
\end{proof}

The following results illustrate the nature of the interpolation operators between two quasi shift-invariant spaces.

\begin{proposition}\label{prop_recovery_1}
Under the assumptions of Theorem \ref{thm_interp_1}, the following holds:
\[
\| \widehat{I_{\phi}^\Y f}  \|_{L_2(\T)}\leq (1+C_\Y^4C_\phi) \| (I+B_\psi M_\psi)\widehat{f}    \|_{L_2(\T)},\quad f\in V(\psi,\X).
\]

Consequently,
\[ \|\widehat{I_\phi^\Y f}\|_{L_2(\T)}\leq (1+C_\Y^4C_\phi)(1+C_\X^2C_\Y^2C_\psi)\|\widehat{f}\|_{L_2(\T)},\quad f\in V(\psi,\X).\]
\end{proposition}
\begin{proof}
Note that Lemma \ref{lem_recovery_1}, the triangle inequality, and the fact that $M_{\phi}\widehat{I_{\phi}^\Y f}=\delta_{\phi}u$ provide us with the estimate
\[
\| \widehat{I_{\phi}^\Y f}  \|_{L_2(\T)}\leq \| (I+B_{\psi} M_{\psi})\widehat{f}  \|_{L_2(\T)} + \delta_{\phi}\| \widetilde{B}_{\phi}u     \|_{L_2(\T)}.
\]
Now the second term is majorized by $\delta_\phi C_\Y^4C_\phi\|u\|_{L_2(\T)}$ on account of \eqref{EQTildeBphiBound}, whence it suffices to show that
\begin{equation}\label{EQProp2}
\delta_{\phi}\| u     \|_{L_2(\T)}\leq  \| (I+B_{\psi} M_{\psi})\widehat{f}  \|_{L_2(\T)}.
\end{equation}
Rewriting the result of Lemma \ref{lem_recovery_1} as 
\[ \widehat\phi u + \delta_\phi\widetilde{B}_\phi u = (I+B_\psi M_\psi)\widehat f,\]
then taking the inner product with $u$ and appealing to (A1) and (A2), the Cauchy-Schwarz inequality, and positivity of $\widetilde{B}_{\phi}$, yields
\[
\langle \widehat{\phi}u, u  \rangle \leq \| (I+B_\psi M_\psi)\widehat{f}  \|_{L_2(\T)} \|u  \|_{L_2(\T)}.
\]

Now the elementary inequality $\delta_{\phi}\| u \|_{L_2(\T)}^2\leq \langle \widehat{\phi}u, u  \rangle $ combined with the previous estimate yields \eqref{EQProp2}, which completes the proof of the first inequality.  The second statement follows directly from the fact that $\|M_\psi\|\leq1$ and \eqref{EQBphiBound}.
\end{proof}

Using these bounds, we may now estimate the operator norm of the interpolation operator $I_\phi^\Y:V(\psi,\X)\to V(\phi,\Y)$.

\begin{corollary}
With the notation and assumptions above, the following holds for all $f\in V(\psi,\X)$:
\[ \|I_\phi^\Y f\|_{L_2(\R)}\leq \left(1+C_\Y^4C_\phi \right)(1+C_\Y^4C_\phi^2 )^{\frac{1}{2}}(1+C_\X^2C_\Y^2C_\psi)\|\widehat{f}\|_{L_2(\T)}.\]
\end{corollary}

\begin{proof}
Since $I_\phi^\Y f \in V(\phi,\Y)$, we simply combine Propositions \ref{lem_psi_1}\textit{(iv)} and \ref{prop_recovery_1}.
\end{proof}

\section{Recovery Criteria}\label{sect_recovery}

Having determined when interpolation of functions in $V(\psi,\X)$ is possible via functions in $V(\phi,\Y)$ in Section \ref{sect_interp}, we now turn to some approximate sampling schemes which allow for recovery of $f\in V(\psi,\X)$ in a limiting sense from its interpolants in a family of spaces $(V(\phi_\alpha,\Y))_{\alpha\in A}$.  The idea is that while the generator $\psi$ may be complicated, or decay slowly as sinc does, it may be replaced by an interpolating generator $\phi_\alpha$ which gives approximate recovery in both $L_2$ and $L_\infty$, but which may have a much simpler structure.  More specifically, we consider the following problem:

\begin{problem}
Given $\psi$ satisfying (A1) and (A2) and a CIS $\X$, find conditions on a family of interpolating generators $\Phi:=(\phi_\alpha)_{\alpha\in A}$ and CISs $\Y$ such that for every $f\in V(\psi,\X)$, its interpolants $I_{\phi_\alpha}^\Y f\in V(\phi_\alpha,\Y)$ converge to $f$ in $L_2$ and uniformly.
\end{problem}


\subsection{Preliminaries}

 With the above considerations in mind, consider the following criteria:
\begin{enumerate}
\item[(B1)] For every $\alpha\in A$, $\phi_\alpha$ satisfies (A1) and (A2).
\item[(B2)] $C_\Phi:=\sup_{\alpha\in A}\;C_{\phi_\alpha}<\infty$, where $C_{\phi_\alpha}$ is as in condition (A2). 
\item[(B3)] $\inflim{\alpha}\sum_{k\neq0}\|T_{\psi,k}A_{\X}^k(M_\psi-M_{\phi_\alpha})g\|_{L_2(\T)}=0$ for every $g\in L_2(\T)$.
\item[(B4)] $\inflim{\alpha}\sum_{k\neq0}\|(T_{\phi_\alpha,k}A_{\Y}^k-T_{\psi,k}A_{\X}^k)M_{\phi_\alpha}g\|_{L_2(\T)}=0$ for every $g\in L_2(\T)$.
\end{enumerate}

Let us stress that while (B3) and (B4) may seem rather abstruse at the moment, stronger hypotheses may be used which imply these conditions but which nonetheless give rise to many examples and additionally are more easily verified in practice.  We will discuss these in more detail in Section \ref{SECExamples}; presently we turn our attention to consequences of these criteria.  Note that if $\X=\Y=\Z$, then these conditions are much easier to handle since $A_\X^k$ and $A_\Y^k$ are the identity on $L_2(\T)$.

Condition (B1) implies that for all $\alpha$, $M_{\phi_\alpha}\in\mathcal{B}(L_2(\T))$, and additionally
\begin{equation}\label{eqn_recovery_2}
\|M_{\phi_\alpha}\|\leq1,\quad \alpha\in A.
\end{equation}

Together, (B1) and (B2) show that if $k\neq 0 $, $T_{\phi_\alpha,k}\in\mathcal{B}(L_2(\T))$, with
$\|T_{\phi_\alpha,k}\|\leq C_{\phi_\alpha}\leq C_\Phi,$
and thus are bounded independently of $\alpha$ and $k$.  Moreover, \eqref{EQTphiSumBound} implies that \[\sum_{k\neq0}\|T_{\phi_\alpha,k}g\|_{L_2(\T)}\leq C_\Phi\|g\|_{L_2(\T)},\quad \alpha\in A.\]
Likewise, we have $\|B_{\phi_\alpha}\|\leq C_\Y^2C_\X^2C_\Phi$ and $\|\widetilde{B}_{\phi_\alpha}\|\leq C_\Y^4C_\Phi$.  

If $k=0$, (A2), and hence (B1), provides the bound $\|T_{\phi_\alpha,0}\|\leq \delta^{-1}_{\phi_\alpha}\| \widehat{\phi_\alpha} \|_{L_\infty(\T)}$;  however this bound could well depend on $\alpha$, as will be made more clear in the examples section that follows.

\subsection{Recovery}

With these notions in hand, we are ready to demonstrate our main recovery results.  Let us first note that Lemma \ref{lem_recovery_1} and Proposition \ref{prop_recovery_1} imply the following.

\begin{corollary}\label{cor_recovery_1}
The operators $(I+\widetilde{B}_{\phi_\alpha} M_{\phi_\alpha}):L_{2}(\T)\to L_{2}(\T)$ are invertible and the norms of the inverses are bounded independent of $\alpha$.
\end{corollary}

Additionally, we have the following.

\begin{lemma}\label{LEMSOT}
\emph{(B1)--(B4)} imply that $\widetilde{B}_{\phi_\alpha}M_{\phi_\alpha}-B_\psi M_\psi\to0$ in the SOT on $\mathcal{B}(L_2(\T))$.
\end{lemma}
\begin{proof}
Let $g\in L_2(\T)$.  Then
\begin{displaymath}
\begin{array}{lll}
\|(\widetilde{B}_{\phi_\alpha}M_{\phi_\alpha}-B_\psi M_\psi)g\|_{L_2(\T)} & \leq & C^2_{\Y}\underset{k\neq0}\dsum\|(T_{\phi_\alpha,k}A_{\Y}^kM_{\phi_\alpha}-T_{\psi,k}A_{\X}^k M_{\psi})g\|_{L_2(\T)}\\ 
& \leq & C^2_{\Y}\underset{k\neq0}\dsum\|(T_{\phi_\alpha,k}A_{\Y}^k-T_{\psi,k}A_{\X}^k)M_{\phi_\alpha}g\|_{L_2(\T)} \\ & &+  C^2_{\Y}\underset{k\neq0}\dsum\|T_{\psi,k}A_{\X}^k(M_{\phi_\alpha}-M_{\psi})g\|_{L_2(\T)},\\
\end{array}
\end{displaymath}
and both terms converge to 0 as $\alpha\to\infty$ on account of (B3) and (B4).  The first inequality above stems from the uniform bounds on the prolongation operators $A^{k}_{\Y}$ and their adjoints.
\end{proof}

It should be noted that in general, $M_{\phi_\alpha}$ need {\em not} converge to $M_\psi$ in the SOT on $\mathcal{B}(L_2(\T))$.  An example of this is provided by regular interpolators discussed in Section \ref{SECExamples}.  We may now prove our main recovery result.

\begin{theorem}\label{thm_recovery_1}
Suppose that $\X$ and $\Y$ are CISs for $PW_\pi$, $\psi$ satisfies \emph{(A1)} and \emph{(A2)}, and $(\phi_\alpha)_{\alpha\in A}$ satisfies \emph{(B1)--(B4)}. Then for every $f\in V(\psi,\X)$,
\[
\lim_{\alpha\to\infty}\| f - I_{\phi_\alpha}^\Y f   \|_{L_2(\mathbb{R})}=0,
\]
where $I_{\phi_\alpha}^\Y f$ is the unique element of $V(\phi_\alpha,\Y)$ which interpolates $f$ at $\Y$.
\end{theorem}
\begin{proof}
Plancherel's Identity allows us to check this result in the Fourier domain; thus we estimate
\begin{align*}
\| \widehat{f} - \widehat{I_{\phi_\alpha}^\Y f}   \|_{L_2(\mathbb{R})}\leq & \| \widehat{f} - \widehat{I_{\phi_\alpha}^\Y f}   \|_{L_2(\T)}+ \| \widehat{f} - \widehat{I_{\phi_\alpha}^\Y f}   \|_{L_2(\mathbb{R}\setminus\T)}\\
=&: I_1+I_2.
\end{align*}
We begin with $I_1$.  On $\T$, we have
\begin{align*}
\widehat{f}-\widehat{I_{\phi_\alpha}^\Y f} =& \left(I-(I+\widetilde{B}_{\phi_\alpha}M_{\phi_\alpha})^{-1}(I+B_{\psi}M_{\psi})\right)\widehat{f}\\
=&(I+\widetilde{B}_{\phi_\alpha}M_{\phi_\alpha})^{-1}(\widetilde{B}_{\phi_\alpha}M_{\phi_\alpha}-B_{\psi}M_{\psi})\widehat{f}.
\end{align*}
Therefore, Corollary \ref{cor_recovery_1} and Lemma \ref{LEMSOT} imply that $I_1\to0$ as $\alpha\to\infty$.

Next, notice that $I_2$ is majorized by 
$$\sum_{k\neq0}\|T_{\psi,k}A_{\X}^kM_\psi\widehat{f}-T_{\phi_\alpha,k}A_{\Y}^kM_{\phi_\alpha}\widehat{I_{\phi_\alpha}^\Y f}\|_{L_2(\T)},$$ which, in turn, is majorized by $I_{2,1}+I_{2,2}+I_{2,3}$, where
$$I_{2,1}:= \sum_{k\neq0}\|T_{\psi,k}A_{\X}^kM_\psi(\widehat{f}-\widehat{I_{\phi_\alpha}^\Y f})\|_{L_2(\T)},$$
$$I_{2,2}:=\sum_{k\neq0}\|T_{\psi,k}A_{\X}^k(M_\psi-M_{\phi_\alpha})\widehat{I_{\phi_\alpha}^\Y f}\|_{L_2(\T)},$$
and
$$I_{2,3} := \sum_{k\neq0}\|(T_{\psi,k}A_{\X}^k-T_{\phi_\alpha,k}A_{\Y}^k)M_{\phi_\alpha}\widehat{I_{\phi_\alpha}^\Y f}\|_{L_2(\T)}.$$

Notice that $I_{2,1}\leq C_\X^2C_\psi\|\widehat{f}-\widehat{I_{\phi_\alpha}^\Y f}\|_{L_2(\T)}$, which converges to 0 as this is the $I_1$ term above.  Subsequently, $I_{2,2}$ is majorized by 
$$\sum_{k\neq0}\|T_{\psi,k}A_{\X}^k(M_\psi-M_{\phi_\alpha})(\widehat{I_{\phi_\alpha}^\Y f}-\widehat{f})\|_{L_2(\T)}+\sum_{k\neq0}\|T_{\psi,k}A_{\X}^k(M_\psi-M_{\phi_\alpha})\widehat{f}\|_{L_2(\T)},$$
where the first term is majorized by a constant multiple of $\|\widehat{f}-\widehat{I_{\phi_\alpha}^\Y f}\|_{L_2(\T)}$, which is $I_1$, hence converges to 0, and the second term converges to 0 as a result of (B3).

Finally, $I_{2,3}$ is bounded by \begin{multline}
\sum_{k\neq0}\|(T_{\psi,k}A_{\X}^k-T_{\phi_\alpha,k}A_{\Y}^k)M_{\phi_\alpha}(\widehat{I_{\phi_\alpha}^\Y f}-\widehat{f})\|_{L_2(\T)} \\ +\sum_{k\neq0}\|(T_{\psi,k}A_{\X}^k-T_{\phi_\alpha,k}A_{\Y}^k)M_{\phi_\alpha}\widehat{f}\|_{L_2(\T)}.
\nonumber\end{multline}  The first term is bounded above by a constant multiple of $\|\widehat{f}-\widehat{I_{\phi_\alpha}^\Y f}\|_{L_2(\T)}$ by a similar argument to the first term related to $I_{2,2}$, hence converges to 0.  The second term converges to 0 on account of (B4).  Putting these estimates together, we conclude that $I_2\to0$ as $\alpha\to\infty$, whence the conclusion of the theorem.

\end{proof}

\begin{corollary}\label{cor_recovery_4}
With the notations and assumptions of Theorem \ref{thm_recovery_1}, 
\[
\lim_{\alpha\to\infty}| f(x) - I_{\phi_\alpha}^\Y f(x)|=0
\]
uniformly on $\mathbb{R}$ for every $f\in V(\psi,\X)$.
\end{corollary}
\begin{proof}
The proof follows directly from the proof of Theorem \ref{thm_recovery_1} by noticing that $\|f-I_{\phi_\alpha}^\Y f\|_{L_\infty(\R)}\leq\|\widehat{f}-\widehat{I_{\phi_\alpha}^\Y f}\|_{L_1(\R)}$, which by periodization and the Cauchy-Schwarz inequality, is majorized by a constant multiple of
$$\|\widehat{f}-\widehat{I_{\phi_\alpha}^\Y f}\|_{L_2(\T)}+\dsum_{k\neq0}\|T_{\psi,k}A_{\X}^kM_\psi\widehat{f}-T_{\phi_\alpha,k}A_{\Y}^kM_{\phi_\alpha}\widehat{I_{\phi_\alpha}^\Y f}\|_{L_2(\T)}.$$  The first term above converges to $0$ by Theorem \ref{thm_recovery_1}, whilst the second is handled as in the proof thereof.

\end{proof}

For additional convergence phenomena similar to the ones listed in this section, the interested reader is referred to \cite{Hamm,Hamm_Zonotopes,HMNW,L1,Ledford_Bivariate,LMSpline,Madych,siva}.

\section{Examples}\label{SECExamples}
Here we will provide a few examples of the phenomena described above.  Such concrete considerations will also lead us to some alternative statements of the criteria (B3) and (B4) which may be more readily checked in practice.  Our examples are broken up into three cases based on the support of $\widehat\psi$.  The first is when $\supp(\widehat\psi)=\T$ (note that condition (A2) implies that $\supp(\widehat\psi)\supseteq\T$).  The second case is $\supp(\widehat\psi)\subset[-A,A]$ for some $\pi<A<\infty$, and the final case is when the support is not contained in any bounded interval.

\subsection{Case $\supp(\widehat\psi)=\T$ -- Regular Interpolators}
To begin, let us demonstrate that the more general setting here indeed recovers the specific considerations of \cite{L1}; namely, for the case of $V(\sinc,\X)=PW_\pi$, the criteria (B1)--(B4) imply the criteria therein of so-called {\em regular interpolators}.

A family $(\phi_\alpha)_{\alpha\in A}$ is said to be a family of {\em regular interpolators for} $PW_\pi$ provided (B1) and (B2) are satisfied, and in addition, 
\begin{equation}\label{EQRegularInterpolator}
\inflim{\alpha}\dfrac{\delta_{\phi_\alpha}}{\widehat{\phi_\alpha}}=0,\quad\textnormal{a.e. on } \T.
\end{equation}

Note that, in particular, \eqref{EQRegularInterpolator} implies that $M_{\phi_\alpha}\to0$ in the SOT on $\mathcal{B}(L_2(\T))$ as $\alpha\to\infty$ on account of \eqref{eqn_recovery_2} and the dominated convergence theorem.

\begin{theorem}\label{THMRegInt}  If $(\phi_\alpha)$ is a family of regular interpolators for $PW_\pi$ and $\psi$ satisfies \emph{(A1)} and \emph{(A2)} and $\supp(\psihat)=\T$, then $(\phi_\alpha)_{\alpha\in A}$ satisfies \emph{(B1)--(B4)}.  In particular, Theorem \ref{thm_recovery_1} and Corollary \ref{cor_recovery_4} are valid for such families.
\end{theorem}

\begin{proof}
By definition, (B1) and (B2) are satisfied.  Notice also that condition (B3) is vacuous because of the assumption on the support of $\widehat{\psi}$.  Finally, by the now familiar argument, (B4) may be shown as follows:
$$\dsum_{k\neq0}\|T_{\phi_\alpha,k}A_{\Y}^kM_{\phi_\alpha}g\|_{L_2(\T)}\leq C_\Phi C_\Y^2\|M_{\phi_\alpha}g\|_{L_2(\T)},$$ which tends to $0$ as $\alpha\to\infty$ via the observation that $M_{\phi_\alpha}\to0$ in the SOT on $\mathcal{B}(L_2(\T))$.

\end{proof}

Consequently, Theorem \ref{THMRegInt} recovers Theorem 1 of \cite{L1}, which constitutes the special case when $\psi$ is the $\sinc$ function, whose Fourier transform is the characteristic function of $\T$.  Also regard that the proof follows from Proposition \ref{PROPSuppPsi}\textit{(i)} in this case given the assumptions on $\psi$.  Examples of regular interpolators may be found in Section 5 of \cite{L1} and Section 8 of \cite{Hamm}, but here we note that a prominent example is the family of Gaussian generators: $\phi_\alpha(x)=e^{-|x/\alpha|^2}$, $\alpha\geq1$.

\subsection{A Class of Convolution Examples for Bandlimited $\psi$}

Here, we focus on the special case when $\psi$ is a generator satisfying (A1) and (A2) such that $\supp(\widehat\psi)\subset[-A,A]$ for some $\pi< A<\infty$.  Note that if $\widehat\psi\in W(L_\infty,\ell_1)$ and $\widehat\psi$ is compactly supported on $[-A,A]$, then $\psi\in PW_A$.  However, not everything in $PW_A$ is the Fourier transform of a function in $W(L_\infty,\ell_1).$  Indeed, consider a function $f$ such that $\widehat f$ is unbounded on $\T$, e.g. $f$ where $\widehat f(\xi) = \xi^{-1/4}\chi_{(0,\pi)}(\xi).$ This $f$ is in $PW_\pi$ because its Fourier transform is square-integrable, but $\widehat f$ fails to be in the amalgam space.

Additionally, we assume that the interpolation set $\Y$ coincides with the set $\X$ defining the space $V(\psi,\X)$ (the reason for this is explained in Section \ref{SECXNotY}). Suppose that $(\phi_\alpha)_{\alpha\in A}$ is a family of generators, each of which satisies (A1) and (A2).   Let
$$\tau_\alpha:=\phi_\alpha\ast\psi,$$
and note that Proposition \ref{PROPAmalgamBasic} implies that the convolution theorem holds, i.e. $\widehat{\tau_\alpha}=\widehat{\phi_\alpha}\widehat\psi$.  Then we have the following:

\begin{proposition}\label{PROPPrimeConditions}
Suppose $\psi$ satisfies \emph{(A1)} and \emph{(A2)}, $\supp(\widehat\psi)\subset[-A,A]$ for some $\pi<A<\infty$, and $\Y=\X$ is a CIS for $PW_\pi$.  Let $N:=\ceiling{\frac{1}{2\pi}A}$.  If $(\phi_\alpha)_{\alpha\in A}$ satisfies \emph{(B1)},
$$\text{\em (B2')}\quad \underset{\alpha\in A}\sup\;\delta_{\phi_\alpha}^{-1}\|\widehat{\phi_\alpha}\|_{ W(L_\infty,\ell_1(\{-N,\dots,-1,1,\dots,N\})}\leq C,$$
and 
$$\text{\em (B3')}\quad \inflim{\alpha}\left\|\delta_{\phi_\alpha}^{-1}\widehat{\phi_\alpha}-\dfrac{\delta_{\phi_\alpha}^{-1}\delta_\psi^{-1}}{\delta_{\tau_\alpha}^{-1}}\right\|_{W(L_\infty,\ell_1(\{-N,\dots,N\}))}=0,$$
then $(\tau_\alpha)_{\alpha\in A}$ satisfies \emph{(B1)--(B4)}, where $\tau_\alpha=\phi_\alpha\ast\psi$.  Consequently, Theorem \ref{thm_recovery_1} and Corollary \ref{cor_recovery_4} hold for $(\tau_\alpha)_{\alpha\in A}$.
\end{proposition}

\begin{proof}
Note that (A1) for $\tau_\alpha$ follows 
from the fact that (A1) holds for $\phi_\alpha$ and $\psi$.  For (A2), note that
\begin{equation}\label{EQdelta_inequality}
\delta_{\tau_\alpha}\geq\delta_{\phi_\alpha}\delta_\psi.
\end{equation}
Thus, to check (B2) for $(\tau_\alpha)$, we need only notice that $\|\widehat{\tau_\alpha}(\cdot+2\pi k)\|_{L_\infty(\T)}\leq\|\widehat{\phi_\alpha}(\cdot+2\pi k)\|_{L_\infty(\T)}\|\widehat\psi(\cdot+2\pi k)\|_{L_\infty(\T)}$, whence applying \eqref{EQdelta_inequality} and the fact that for $\ell_1$ sequences, $\|ab\|_{\ell_1}\leq\|a\|_{\ell_1}\|b\|_{\ell_1}$, yields
$$\delta_{\tau_\alpha}^{-1}\|\widehat{\tau_\alpha}\|_{W'}\leq\delta_{\phi_\alpha}^{-1}\delta_\psi^{-1}\|\widehat{\phi_\alpha}\|_{W'}\|\widehat\psi\|_{W'},$$
which is bounded by a constant $C$ independent of $\alpha$ on account of the fact that $(\phi_\alpha)$ satisfies (B2') (here we have abbreviated the amalgam space in question to $W'$ for brevity).

Next we check (B3) and (B4).  Notice that 
\begin{displaymath}
\begin{array}{lll}
|\delta_{\tau_\alpha}^{-1}\widehat{\tau_\alpha}-\delta_\psi^{-1}\widehat{\psi}| & = & |\delta_{\tau_\alpha}^{-1}\widehat{\phi_\alpha}\widehat{\psi}-\delta_\psi^{-1}\widehat{\psi}|\\
\\
& = & |\widehat{\psi}|\delta_{\tau_\alpha}^{-1}\left|\widehat{\phi_\alpha}-\frac{\delta_\psi^{-1}}{\delta_{\tau_\alpha}^{-1}}\right|\\
\\
& \leq & |\widehat{\psi}|\delta_{\phi_\alpha}^{-1}\delta_{\psi}^{-1}\left|\widehat{\phi_\alpha}-\frac{\delta_\psi^{-1}}{\delta_{\tau_\alpha}^{-1}}\right|\\
\\
& \leq & \delta_\psi^{-1}|\widehat{\psi}|\left|\delta_{\phi_\alpha}^{-1}\widehat{\phi_\alpha}-\frac{\delta_{\phi_\alpha}^{-1}\delta_\psi^{-1}}{\delta_{\tau_\alpha}^{-1}}\right|,\\
\end{array}
\end{displaymath}
where the first inequality follows from \eqref{EQdelta_inequality}.

Consequently, letting $W:=W(L_\infty,\ell_1(\{-N,\dots,N\}))$,
$$\|\delta_{\tau_\alpha}^{-1}\widehat{\tau_\alpha}-\delta_\psi^{-1}\widehat\psi\|_{W}\leq \delta_\psi^{-1}\|\widehat\psi\|_{W}\left\|\delta_{\phi_\alpha}^{-1}\widehat{\phi_\alpha}-\frac{\delta_{\phi_\alpha}^{-1}\delta_\psi^{-1}}{\delta_{\tau_\alpha}^{-1}}\right\|_{W}.$$
To conclude the proof, it suffices to notice that the above inequality together with (B3') implies both (B3) and (B4) for $(\tau_\alpha)$.  In particular, (B3') implies that $M_{\tau_\alpha}\to M_{\psi}$ in the SOT on $\mathcal{B}(L_2(\T))$. 
\end{proof}

\begin{remark}\label{REMB3}
As a special case of this, suppose that $\delta_{\tau_\alpha}=\delta_{\phi_\alpha}\delta_\psi$, which can happen e.g. if $\widehat{\phi_\alpha}$ and $\widehat{\psi}$ are even and decreasing on $[0,\pi]$.  In this case, \emph{(B3')} reduces to the statement that $\|\delta_{\phi_\alpha}^{-1}\widehat{\phi_\alpha}-1\|_{W(L_\infty,\ell_1(\{-N,\dots,N\}))}\to0$ as $\alpha\to\infty$.
\end{remark}

\begin{remark}
It should also be noticed that a family $(\phi_\alpha)$ satisfying \emph{(B2')} need not satisfy \emph{(B2)} as the following example will illustrate.  That is, $\delta_{\phi_\alpha}^{-1}\|\widehat{\phi_\alpha}\|_{W(L_\infty,\ell_1')}$ need not be uniformly bounded for $\alpha\in A$ since the proof of Proposition \ref{PROPPrimeConditions} only requires a finite number of terms in the amalgam norm.
\end{remark}

\begin{remark}\label{REMLinfty}
In Proposition \ref{PROPPrimeConditions}, the assumption that $(\phi_\alpha)$ satisfies (B1) can be relaxed.  If for all $\alpha$, we have $\widehat{\phi_\alpha}\in L_\infty(\mathbb{R})$, with $\widehat{\phi_\alpha}(\xi)\geq 0$ on $\mathbb{R}$ and $\widehat{\phi_\alpha}(\xi)>0$ on $\T$.  Then $(\tau_\alpha)$ clearly satisfies (A1) and (A2), and we have
\[
C_{\tau_\alpha}\leq \delta^{-1}_{\phi_\alpha}\|\widehat\phi_\alpha \|_{L_\infty(\mathbb{R})}C_\psi,
\]
thus to account for (B2), we impose the additional hypothesis that
\[
\widetilde{C}_{\Phi}:=\sup_{\alpha\in A}\delta^{-1}_{\phi_\alpha}\|\widehat\phi_\alpha \|_{L_\infty(\mathbb{R})} <\infty.
\]
For an illustration of this, see Example \ref{EXApproximateIdentity} below. 
\end{remark}

\begin{example}[Convolution with the Poisson Kernel]\label{EXPoissonConvolution}
Let $\tau_\alpha(x) := e^{-\alpha|\cdot|}\ast\psi$ with $\psi$ such that $\delta_{\tau_\alpha}=\delta_{\phi_\alpha}\delta_\psi$ for each $\alpha$.   Then $$\widehat{\tau_\alpha}(\xi) = \widehat{\phi_\alpha}(\xi)\widehat{\psi}(\xi) = \sqrt{\frac{2}{\pi}}\frac{\alpha}{\alpha^2+\xi^2}\widehat{\psi}(\xi).$$  Evidently, $\delta_{\phi_\alpha} = \sqrt{\frac{2}{\pi}}\frac{\alpha}{\alpha^2+\pi^2}$, and for $k\neq 0$, $$\|\widehat{\phi_\alpha}(\cdot+2\pi k)\|_{L_\infty(\T)} = \sqrt{\frac{2}{\pi}}\frac{\alpha}{\alpha^2+(2|k|-1)^2\pi^2}.$$  Therefore,
\begin{equation}\label{EQPoissonDelta}\delta_{\phi_\alpha}^{-1}\underset{k\neq0}\dsum\|\widehat{\phi_\alpha}(\cdot+2\pi k)\|_{L_\infty(\T)} = \underset{k\neq0}\dsum \dfrac{\alpha^2+\pi^2}{\alpha^2+(2|k|-1)^2\pi^2}.
\end{equation}
Note that the series on the right hand side of \eqref{EQPoissonDelta} is increasing as $\alpha$ increases, and moreover each term tends to 1 as $\alpha\to\infty$.  Consequently, $\delta_{\phi_\alpha}^{-1}\|\widehat{\phi_\alpha}\|_W$ is not bounded above by a constant independent of $\alpha$.  However, $\delta_{\tau_\alpha}^{-1}\|\widehat{\tau_\alpha}\|_W$ is since $\widehat\psi$ is compactly supported.  Indeed, if $\supp(\widehat{\psi})\subset[-N,N]$, then \eqref{EQPoissonDelta} implies that
$$\inflim{\alpha}\delta_{\phi_\alpha}^{-1}\finsum{k}{-N}{N}\|\widehat{\phi_\alpha}(\cdot+2\pi k)\|_{L_\infty(\T)}= 2N+1,$$ which implies (B2').

To verify (B3'), we use Remark \ref{REMB3} and simply notice that $\inflim{\alpha}\frac{\alpha^2+\pi^2}{\alpha^2+\xi^2}=1$ uniformly in $\xi$ on $\T+2\pi k$ for any $k\in\{-N,\dots,N\}$.

\end{example}

\begin{example}[Convolution with the Gaussian Kernel]\label{EXGaussianConvolution}
Similar to Example \ref{EXPoissonConvolution}, we obtain the same convergence results when $\tau_\alpha = e^{-\alpha|\cdot|^2}\ast\psi$, letting $\alpha\to\infty$.  It should be noted that the parameter here is the opposite as in the regular interpolators case, where the Gaussian parameter limits to $0$.  In both of these examples, convolution is used to smooth out the generator with something that decays rapidly.
\end{example}

\begin{example}[Convolution with Inverse Multiquadrics]
In the vein of Examples \ref{EXPoissonConvolution} and \ref{EXGaussianConvolution}, consider $\tau_\alpha:=\phi_\alpha\ast\psi$ where $\phi_\alpha(x):=(x^2+1)^{-\alpha}$ is the inverse multiquadric of exponent $\alpha$, and we let $\alpha\to\infty$.  Then from \cite{Jones},
$$\widehat{\phi_{\alpha}}(\xi) = \sqrt{2\pi}\frac{2^{1-\alpha}}{\Gamma(\alpha)}|\xi|^{\alpha-\frac{1}{2}}K_{\alpha-\frac{1}{2}}(|\xi|),$$
where $K_\nu$ is the modified Bessel function of the second kind (see \cite[p. 376]{AS} for the precise definition).

Since $\widehat{\phi_{\alpha}}$ is decreasing, $\delta_{\tau_\alpha}=\widehat{\phi_\alpha}(\pi)$ as before.  The other conditions being easily checked, let us consider (B3').  It suffices to check for $k=-N,\dots,N$ that $|\delta_{\phi_\alpha}^{-1}\widehat{\phi_\alpha}(\xi+2\pi k)-1|\to0$ uniformly for $\xi\in\T$.  Notice that
$$\dfrac{\widehat{\phi_\alpha}(\xi+2\pi k)}{\widehat{\phi_\alpha}(\pi)} = \dfrac{|\xi+2\pi k|^{\alpha-\frac{1}{2}}K_{\alpha-\frac{1}{2}}(|\xi+2\pi k|)}{\pi^{\alpha-\frac{1}{2}}K_{\alpha-\frac{1}{2}}(\pi)}.$$
To consider the limit as $\alpha\to\infty$, we need to know the asymptotic behavior of the modified Bessel function of the second kind with respect to its order.  From \cite{Sidi}, we find that
\[
K_\nu(z) \sim 2^{\nu-1}\Gamma(\nu)z^{-\nu},\quad\nu\to\infty,
\]
where the notation $f(x)\sim g(x)$, $x\to\infty$ means that $\inflim{x}f(x)/g(x)=1$. Thus, setting $\nu=\alpha-1/2$, we have that
$$\dfrac{\widehat{\phi_\alpha}(\xi+2\pi k)}{\widehat{\phi_\alpha}(\pi)}\sim \dfrac{|\xi+2\pi k|^\nu2^{\nu-1}\Gamma(\nu)|\xi+2\pi k|^{-\nu}}{\pi^\nu2^{\nu-1}\Gamma(\nu)\pi^{-\nu}} = 1,\quad \nu\to\infty,$$
which yields (B3').
\end{example}

\begin{example}[Convolution with Approximate Identities]\label{EXApproximateIdentity}
Finally, we may consider a large class of approximate identities.  These are similar in spirit to the examples presented thus far, but the previous examples are not necessarily approximate identities of the form considered here.  Again make the assumption that $\widehat\psi$ is compactly supported on $[-A,A]$, with $N:=\ceiling{A}$, and that $\X=\Y$.  For our purposes, an \textit{approximate identity} is a function $\phi$ with $\widehat\phi>0$ on $\R$, with $\phi,\widehat\phi\in L_1$, and $\int_\R\phi(x)dx=1$. 
Then set $\phi_\alpha(x):=\alpha\phi(\alpha x)$, and we find that $\tau_\alpha:=\phi_\alpha\ast\psi$ satisfies (B1)--(B4).

To check that $(\tau_\alpha)_{\alpha\geq1}$ satisfies (B1)--(B4), first note that by Remark \ref{REMLinfty}, (B1) is satisfied; therefore we simply verify that $(\phi_\alpha)_{\alpha\geq1}$ satisfies (B2') and (B3').  Recalling that $\widehat{\phi_\alpha}(\xi)=\widehat\phi(\xi/\alpha)$, it follows that the quantity $\delta_{\phi_\alpha}:=\underset{\xi\in\T}\inf|\widehat{\phi}(\xi/\alpha)|$ is non-decreasing as $\alpha$ increases, and hence $\delta_{\phi_\alpha}^{-1}\leq\delta_{\phi}^{-1}.$

Next, note that since $\widehat\phi\in L_1$, the inversion formula holds, and we have that $|\widehat\phi(\xi)|\leq\|\phi\|_{L_1}$ for almost every $\xi\in\R$.  Consequently,  $$\finsum{j}{-N}{N}\left\|\widehat{\phi}\left(\frac{\cdot+2\pi j}{\alpha}\right)\right\|_{L_\infty(\T)}\leq (2N+1)\|\phi\|_{L_1}.$$
Combining this with the previous observation about $\delta_{\phi_\alpha}^{-1}$ yields the conclusion of (B2').

To see (B3'), note that $\delta_{\phi_\alpha}\to\widehat\phi(0)$ as $\alpha\to\infty$, and that for any fixed $j\in\{-N,\dots,N\}$, $\widehat\phi((\xi+2\pi j)/\alpha)\to\widehat\phi(0)$.  Consequently, we have $\|\delta_{\phi_\alpha}^{-1}\widehat{\phi}-1\|_{W}\to0$.  Moreover, the amalgam norm of the constant $\delta_{\tau_\alpha}/(\delta_{\phi_\alpha}\delta_\psi)-1$ goes to $0$ as well by a similar argument.  Thus by the triangle inequality, (B3') is satisfied. 

\end{example}

\subsection{Non-bandlimited $\psi$}

Let us now consider the case when $\widehat\psi$ is not compactly supported, and again $\X=\Y$.  Then if $\tau_\alpha=\phi_\alpha\ast\psi$ where $(\phi_\alpha)$ satisfies (B1) and modified versions of (B2') and (B3') where the amalgam norms therein are taken to be $W(L_\infty,\ell_1')$ and $W(L_\infty,\ell_1)$, respectively, the conclusion of Proposition \ref{PROPPrimeConditions} holds for $(\tau_\alpha)$. It should be noted that in this case, the Gaussian of Example \ref{EXGaussianConvolution} still yields recovery, but convolution with the Poisson kernel does not because then the family $(\phi_\alpha)_{\alpha\in A}$ does not satisfy the uniform bound in (B2') when the sum is infinite (see the discussion in Example \ref{EXPoissonConvolution}). 

\subsection{Interpolation at $\X\neq\Y$}\label{SECXNotY}
It is pertinent to examine the case when $\supp(\widehat\psi)\supsetneq\T$, and $\Y\neq\X$.  Unfortunately, recovery turns out to not be generally feasible, a fact we record in the following proposition.

\begin{proposition}\label{PROPYnotX}
There exist $\psi$ with $\supp(\widehat\psi)\supsetneq\T$, and $\Y\neq\X$ CISs for $PW_\pi$ such that there is no family $(\phi_\alpha)_{\alpha\in A}$ satisfying \emph{(B1)} for which the interpolants $I_{\phi_\alpha}^\Y f\in V(\phi_\alpha,\Y)$ converge in $L_2$ and uniformly to $f$ for all $f\in V(\psi,\X)$.
\end{proposition}

As of yet, we do not have enough information at our disposal to provide the proof, but we return to the matter at the end of Section \ref{sect_cardinal}.

\section{Cardinal Functions}\label{sect_cardinal}

In this section, we analyze the special case when $\mathcal{X}=\Z$, in which case the space $V(\psi):=V(\psi,\Z)$ is called the {\em principal shift-invariant} space associated with the generator $\psi$ -- an  object of extensive study in many areas of harmonic analysis, approximation theory, and functional analysis.   We still make the assumptions (A1) and (A2) on $\psi$, and in what follows, assume that $(\phi_\alpha)_{\alpha\in A}$ is a one-parameter family of generators satisfying (B1)-(B4).  

Cardinal interpolation arose from the penetrating work of I. J. Schoenberg on spline interpolation \cite{SchoenbergQuarterly,Schoenberg}, and from summability methods for the sampling series found in the WKS sampling formula.  Recall that if $f\in PW_\pi$, then \begin{equation}\label{EQSampling}
f(x) = \underset{j\in\Z}\dsum f(j)\,\sinc(x-j),\end{equation} where $\sinc(x) = \frac{\sin(\pi x)}{\pi x}$ if $x\neq0$, and $\sinc(0)=1$.  Convergence of the series in \eqref{EQSampling} was later shown to be both in the sense of $L_2(\R)$ and uniform on $\R$.  But to E. T. Whittaker \cite{Whittaker}, \eqref{EQSampling} was an equation of interpolation, i.e. clearly $f(k) = \sum_{j\in\Z}f(j)\,\sinc(k-j)$ for $k\in\Z$ since $\sinc(n)=\delta_{0,n}$.

However, the sinc series converges slowly in the sense that $\sinc(x)=O(|x|^{-1})$.  Consequently, many authors, including Schoenberg, have sought to replace sinc with another {\em cardinal function} which has the property that $L(n)=\delta_{0,n}$, $n\in\Z$, but which decays more rapidly than sinc, hence the term summability method. There are of course other ways around use of the sinc kernel; for example, the generalized sampling kernels of Butzer, Ries, and Stens \cite{Butzer}, but we restrict our attention here to cardinal function methods. There are many examples of such cardinal functions and their associated decay rates \cite{Baxter,Buhmann,Buhmann_Book,BuhmannMQ,HL1,HL2, Ledford_Cardinal,MadychMisc}.  Of primary interest to us is their construction from a given function as follows.  

Given $\phi$, formally define
\begin{equation}\label{EQCardinalDef}
\widehat{L_\phi}(\xi):=\dpifactor\dfrac{\widehat{\phi}(\xi)}{\underset{j\in\Z}\dsum\widehat{\phi}(\xi+2\pi j)}.
\end{equation}
Under certain conditions (for example, if $\widehat{L_\phi}\in L_1\cap L_2$) the inverse Fourier transform, $L_\phi$, will be a cardinal function which satisfies $L_\phi(k)=\delta_{0,k}$, $k\in\Z$.  Indeed, one needs only justify the following formal calculation:
\begin{align*}
L_\phi(k) & = \dfrac{1}{\sqrt{2\pi}}\underset{j\in\Z}\sum\int_\T \dfrac{\widehat{\phi}(\xi+2\pi j)}{\underset{m\in\Z}\sum\widehat{\phi}(\xi+2\pi m)}e^{i(\xi+2\pi j)k}d\xi\\ & = \dfrac{1}{\sqrt{2\pi}}\int_\T\dfrac{\underset{j\in\Z}\sum\widehat{\phi}(\xi+2\pi j)}{\underset{m\in\Z}\sum\widehat{\phi}(\xi+2\pi m)}e^{i\xi k}d\xi\\ & = \delta_{0,k}.
\end{align*}

Consequently, if the family of interpolators is made from convolution with the generator $\psi$, i.e. $\tau_\alpha(x) = \phi_\alpha\ast\psi(x)$, then the Fourier transform of the cardinal function is
$$\widehat{L_{\tau_\alpha}}(\xi) = \dpifactor\dfrac{\widehat{\phi_\alpha}(\xi)\widehat{\psi}(\xi)}{\underset{j\in\Z}\dsum\widehat{\phi_\alpha}(\xi+2\pi j)\widehat{\psi}(\xi+2\pi j)}.$$

Moving on to more general cardinal functions, there are a couple of natural questions that arise.  The first is, does $L_\phi$ satisfy (A1) and (A2)?  Provided that $\phi$ itself does, then the answer is yes.  We exhibit this in the following proposition.

\begin{proposition}\label{PROPLphiCardinal}
If $\phi$ satisfies \emph{(A1)} and \emph{(A2)}, then $L_\phi = \frac{1}{\sqrt{2\pi}}\int_\R\widehat{L_\phi}(\xi)e^{i\xi \cdot}d\xi$ is a cardinal function, and moreover, $L_\phi$ satisfies \emph{(A1)} and \emph{(A2)}.
\end{proposition}

\begin{proof}
{
Since (A1) and (A2) hold for $\phi$, we have $ \widehat\phi(\xi)\geq 0$ on $\mathbb{R}$ and $\delta_\phi>0$, thus Bochner's theorem and a routine periodization argument show that (A1) holds for $L_\phi$ as well. Additionally, since $\widehat{\phi}$ is nonnegative, the calculation above evaluating $L_\phi(k)$ is valid by the monotone convergence theorem, and so $L_\phi$ given by the Fourier inversion formula is a cardinal function provided we have $\widehat{L_\phi}\in L_1\cap L_2$, which follows from (A1).

For (A2), we have $$\sum_{j\in\Z}\|\widehat{L_\phi}(\cdot+2\pi j)\|_{L_\infty(\T)}\leq\delta_\phi^{-1}\dzsum{j}\|\widehat{\phi}(\cdot+2\pi j)\|_{L_\infty(\T)} = \delta_\phi^{-1}\|\widehat\phi\|_{W}<\infty,$$ which implies that $\widehat{L_\phi}\in W(L_\infty,\ell_1)$.  Now we can calculate $C_{L_\phi}$ by noting $\widehat{L_\phi}\geq\dfrac{\delta_\phi}{\|\widehat\phi\|_{W}}$ on $\T$, which leaves us with
$$C_{L_\phi}\leq C_\phi\left(\delta_\phi^{-1}\|\widehat\phi(\xi)\|_{L_\infty(\T)}+C_\phi  \right)<\infty. $$

}

\end{proof}
Note that Proposition \ref{PROPLphiCardinal} implies that $V(L_\phi)$ is a well-defined shift-invariant space.

\subsection{On Interpolation via Cardinal Functions}

Given cardinal functions constructed previously, we now turn to their interpolation properties.  Proposition \ref{PROPLphiCardinal} together with Theorem \ref{thm_interp_1} implies that interpolation of functions in $V(\psi)$ via interpolants in $V(L_\phi)$ is possible.  We now enumerate some of the consequences of this fact beginning with the following lemma.  For ease of notation in this section, we write $I_\phi$ for $I_\phi^\Z$, representing the interpolation operator from $V(\phi)\to V(\psi)$. 

\begin{lemma}\label{LEMIpsi}
Let $I_\psi$ be the interpolation operator associated with the generator $\psi$.  Then if $f\in V(\psi)$, $f=I_\psi f$.
\end{lemma}

\begin{proof}
Note that by \eqref{eq_interp_1} and Theorem \ref{thm_interp_1}\textit{(i)}, $I_\psi f\in V(\psi)$.  Moreover, $I_\psi f$ is the unique function in $V(\psi)$ such that $I_\psi f(k)=f(k)$.  However, evidently $f\in V(\psi)$ satisfies this relation as well; consequently $I_\psi f=f$.
\end{proof}

This lemma leads us to the following proposition.

\begin{proposition}\label{PROPCardinalConvergence}
If $(\phi_\alpha)$ satisfies \emph{(B1)--(B4)}, then $L_{\phi_\alpha}\to L_\psi$ both uniformly and in $L_2(\R)$ as $\alpha\to\infty$.
\end{proposition}

\begin{proof}
First, note that via Theorem \ref{thm_interp_1} (or Lemma 2 of \cite{L1}), there exists a function $f\in V(\psi)$ such that $f(j)=\delta_{0,j}$.  For this $f$, $I_{\phi_\alpha} f = L_{\phi_\alpha}$, and $f=L_\psi$ by Lemma \ref{LEMIpsi}.  Thus, an application of the conclusion of Theorem \ref{thm_recovery_1} demonstrates that
$$\inflim{\alpha}I_{\phi_\alpha} f = \inflim{\alpha}L_{\phi_\alpha} = f = L_\psi$$ uniformly and in $L_2$.
\end{proof}

Another, perhaps more important question, involves the form of the interpolant $I_\phi f$ to a given $f\in V(\psi)$.  Theorem \ref{thm_interp_1} shows that $I_\phi f = \sum_{j\in\Z}a_j\phi(\cdot-j)$ is the unique element of $V(\phi)$ that interpolates $f$ at the integer lattice.  However, the definition of the cardinal function implies that the following function interpolates $f$ at the integers:
$$\widetilde{I_\phi f}(x):=\dzsum{j}f(j)L_\phi(x-j).$$
Moreover, Lemma \ref{lem_psi_2} implies that $\widetilde{I_\phi f}\in V(L_\phi)$, on account of the interpolatory condition, is the unique element of $V(L_\phi)$ that interpolates $f$ at $\Z$.  
The following theorem is a consequence of the characterization of principal shift-invariant subspaces of $L_2$ in \cite{DeBoorDeVoreRon}, and implies that indeed $\widetilde{I_\phi f} = I_\phi f$.  For completeness we give the proof here.

\begin{theorem}\label{THMShiftSpaceEquality}
If $\phi$ satisfies \emph{(A1)} and \emph{(A2)}, then $V(\phi)=V(L_\phi)$.
\end{theorem}

\begin{proof}
Again, it suffices to show that $\F V(\phi)=\F V(L_\phi)$.  By definition, $\F V(\phi) = \{\sum_{j\in\Z}c_je^{-ij(\cdot)}\widehat{\phi}:(c_j)\in\ell_2\}.$  However, we may equivalently write the space as $\{Q\widehat{\phi}:Q|_{\T}\in L_2(\T),  Q \textnormal{ is } 2\pi\textnormal{--periodic}\}.$  
The proof may be concluded by simply noticing the $\widehat{L_\phi} = \widehat\phi\sigma$ where $\sigma(\xi)=\sum_{j\in\Z}\widehat\phi(\xi+2\pi j)$ is a continuous, $2\pi$--periodic function which is bounded above and below on $\T$.  Consequently, $Q\widehat\phi = (Q/\sigma)\widehat{L_\phi}$ with $Q/\sigma$ a $2\pi$--periodic $L_2(\T)$ function, whence $\F V(\phi) = \F V(L_\phi)$.
\end{proof}

Consequent upon Theorem \ref{THMShiftSpaceEquality}, the unique interpolant of $f$ from the shift-invariant space $V(\phi)$ takes the forms
$$I_\phi f(x)=\sum_{j\in\Z}a_j\phi(x-j)=\sum_{j\in\Z}f(j)L_\phi(x-j).$$

As promised, this section concludes with the counterexample to recovery whenever $\Y\neq\X$.
\begin{proof}[Proof of Proposition \ref{PROPYnotX}]
Let $\Y=\Z$ and $\X=\Z\setminus\{1\}\cup\{\sqrt{2}\}$ ($\Z$ is obviously a CIS, and the perturbation of only finitely many points in a CIS yields another CIS provided the resulting points are pairwise distinct).  Additionally, let $\psi$ be the Gaussian kernel $e^{-|x|^2}$, so that $\supp(\widehat\psi)=\R$ and $\widehat\psi>0$ on $\R$.  

By way of contradiction, suppose that there was a family of generators which satisfy (B1) such that for every $f\in V(\psi,\X)$, we have $\underset{\alpha\to\infty}\lim\;I_{\phi_\alpha}^\Z f= f$ in $L_2$ and uniformly, where $I_{\phi_\alpha}^\Z f\in V(\phi_\alpha,\Z)$.

First, notice that by Theorem \ref{THMShiftSpaceEquality}, $V(\phi_\alpha,\Z) = V(L_{\phi_\alpha},\Z)$.  Therefore, for every $f\in V(\psi,\X)$, we have
$$I_{\phi_\alpha}^\Z f = I^\Z_{L_{\phi_\alpha}}f = \sum_{j\in\Z}f(j)L_{\phi_\alpha}(\cdot-j)$$
via the uniqueness of the interpolant.  Consider first that 
$$\widehat{I^\Z_{L_{\phi_\alpha}}\psi} = \sum_{j\in\Z}\psi(j)e^{-ij\xi}\widehat{L_{\phi_\alpha}}(\xi).$$  By the Poisson Summation Formula (which clearly holds for the Gaussian) this is $$\sum_{k\in\Z}\widehat\psi(\xi+2\pi k)\widehat{L_{\phi_\alpha}}(\xi)=:\sigma_\psi(\xi)\widehat{L_{\phi_\alpha}}(\xi).$$
Thus $\widehat{I^\Z_{L_{\phi_\alpha}}\psi} = \sigma_\psi\widehat{L_{\phi_\alpha}}\to\widehat\psi$ in $L_2$, which implies that $\widehat{L_{\phi_\alpha}}\to\frac{\widehat\psi}{\sigma_\psi} = \widehat{L_\psi}$ in $L_2$ (since $\sigma_\psi$ is a $2\pi$--periodic function that is bounded above and below by positive constants for every $\xi\in\R$).

Therefore, $L_{\phi_\alpha}\to L_\psi$ in $L_2(\R)$.  This implies that if $f_1 = \psi(\cdot-\sqrt{2})$, which is in $V(\psi,\X)$, we have $$I^\Z_{L_{\phi_\alpha}}f_1\to\sum_{j\in\Z}f_1(j)L_\psi(\cdot-j)$$
in $L_2$ (this follows again by using the Poisson Summation Formula on $f_1$, which is evidently valid, and the fact that $\sum_{k\in\Z}|\widehat{f_1}(\xi+2\pi k)|\leq C$ for $\xi\in\T$).  But the right hand side above is in $V(L_\psi,\Z)=V(\psi,\Z)$.  On the other hand, by assumption, $I^\Z_{L_{\phi_\alpha}}f_1=I^\Z_{\phi_\alpha}f_1\to f_1$ which is in $V(\psi,\X)\setminus V(\psi,\Z)$, which yields a contradiction.
\end{proof}

\subsection{Extensions for Cardinal Interpolation}

One of the more interesting utilities of cardinal functions defined as in \eqref{EQCardinalDef} is that they may still be well-defined even when the generator $\phi$ grows.  For example, if $\phi(x):=\sqrt{x^2+c^2}$, which is the traditional Hardy multiquadric \cite{Hardy} then the cardinal function $L_\phi$ is well-defined because $\widehat{\phi}$ may be identified with a function which has an algebraic singularity at the origin and decays exponentially away from the origin \cite{Jones}, thus allowing the right-hand side of \eqref{EQCardinalDef} to be defined for every $\xi\in\R\setminus\{0\}$.  Yet in this case, $V(L_\phi,\Z)\neq V(\phi,\Z)$ because the space associated with $\phi$ is not well-defined because of the growth of the generator.  However, the decay of $L_\phi$ and its Fourier transform is such that its principal shift-invariant space $V(L_\phi)$ is indeed well-defined \cite{Buhmann, Buhmann_Book}.  

Additionally, there do exist families of cardinal functions which satisfy condition (B1); as a canonical example, we consider $(L_{\phi_c})_{c\in[1,\infty)}$, the cardinal functions associated with the Hardy multiquadric mentioned above indexed by the shape parameter $c$.  Suppose for simplicity that $\supp(\widehat{\psi})=\T$.  Then the fact that $(L_{\phi_c})$ satisfies (A1), (A2), and (B2) may be surmised from \cite{RS}, while (B3) is vacuous based on the support of $\widehat\psi$.  Finally, (B4) follows from Proposition 2.2 of \cite{Baxter}.  Thus there are examples of cardinal functions which exhibit convergence by satisfying these conditions despite the fact that the generators they are formed from manifestly do not.  So while often the spaces $V(\phi,\Z)$ and $V(L_\phi,\Z)$ coincide, there is sometimes additional flexibility when using cardinal functions.  

In \cite{Ledford_Cardinal}, sufficient conditions on a family of multivariate generators $(\phi_\alpha)$ were given such that cardinal interpolation from the space $V(L_{\phi_\alpha},\Z^d)$ (defined in the obvious manner for $\Z^d$) is well-defined, and moreover, the interpolants of a bandlimited function converge to that function both in $L_2$ and uniformly on $\R^d$ as $\alpha\to\infty$.

\section{On Inverse Theorems}

The conclusion of our analysis features a discussion of inverse theorems with respect to the generators, or rather lack thereof.  Indeed, it is an interesting question whether convergence of interpolants $I_{\phi_\alpha}^\Y f \to f$ for every $f\in V(\psi,\X)$ implies that in some manner $\phi_\alpha\to\psi$.  In all cases described above, the answer to this question is negative.

The first case to consider is when $\supp(\widehat\psi)=\T$.  In this case, the fact that $I_{\phi_\alpha}^\Y f\to f$ for every $f\in V(\psi,\X)$ does not imply that $\phi_\alpha\to\psi$.  Essentially all regular interpolators of \cite{L1} are counterexamples; in particular, $\phi_\alpha:=e^{-\frac{|\cdot|^2}{\alpha}}$ provides recovery in $V(\sinc,\Z)=PW_\pi$ as $\alpha\to\infty$, but clearly $\phi_\alpha\nrightarrow\sinc$ in any classical (e.g. pointwise, $L_p$, etc.) manner.

For more general $\psi$, consider interpolation of functions in $V(\psi,\Z)$ via $V(\phi_\alpha,\Z)$.  Note that if $I_{\phi_\alpha}^\Z f\to f$, then Theorem \ref{THMShiftSpaceEquality} and uniqueness of the interpolant (Theorem \ref{thm_interp_1}) implies that $I_{L_{\phi_\alpha}}^\Z f\to f$ for all $f$ as well.  But Proposition \ref{PROPCardinalConvergence} implies that $L_{\phi_\alpha}\to L_\psi$, whereas there are many generators for which $\psi\neq L_\psi$ (for example, if $\psi(x)=e^{-|x|^2}$, then $L_\psi\neq\psi$, a fact that can be checked via \eqref{EQCardinalDef}), which means that we cannot also have $L_{\phi_\alpha}\to\psi$ in this case .


\section{Remarks}

While the above analysis thoroughly explores Problem \ref{PROB} in the $L_2$ quasi shift-invariant space, it is natural to consider what happens in the $L_p$ setting for general $p$. Some things carry over in the uniform setting; for instance, under conditions (A1) and (A2), the systems $\{\phi(\cdot-j):j\in\Z\}$ and $\{L_\phi(\cdot-j):j\in\Z\}$ are unconditional bases for their span in $L_p$, which we denote $V_p(\phi,\Z)$ and $V_p(L_\phi,\Z)$, respectively, and moreover they are closed subspaces of $L_p$ \cite{AG}.  Under the additional assumption that the symbol $\sigma(\xi):=\sum_{j\in\Z}\widehat\phi(\xi+2\pi k)$ is in the Wiener algebra $A(\T)$ of $2\pi$--periodic functions with absolutely summable Fourier coefficients, we also have that $V_p(\phi,\Z)=V_p(L_\phi,\Z)$ for $p\in[1,2]$. The proof follows from the representation of the symbol as $\sigma(\xi)=\sum_{j\in\Z}d_je^{-ij\xi}$ with $(d_j)\in\ell_1$ and elementary norm inequalities.

Additionally, the results here involving cardinal functions extend easily to higher dimensions in the case $\X=\Z^d$.  However, for more general $\X\subset\R^d$, the methods here do not extend readily, predominantly due to the fact that Riesz bases of exponentials are difficult to come by in higher dimensions even for straightforward domains (e.g. it is an open problem whether or not a Riesz basis of exponentials exists for the Euclidean ball in $\R^d$ for any $d\geq2$).  For recent results on existence of Riesz bases, consult \cite{grepstad_lev,kol,nitzan}. For some interpolation and recovery schemes similar to those here in higher dimensions, see \cite{BSS,Hamm_Zonotopes,Ledford_Bivariate}.  Naturally, \cite{GS} and others consider quasi shift-invariant spaces for more general point-sets $\X$, which would eliminate this difficulty, but for the interpolation method examined here, the techniques of proof do not extend to sets $\X$ which are merely quasi-uniform, though that is not to say that no such method is feasible.

\section*{Acknowledgments}
The first author thanks Akram Aldroubi, Alex Powell, and Ben Hayes for many fruitful discussions involving this work.  The authors also take pleasure in thanking the anonymous referee for their valuable suggestions which greatly improved this article.



\end{document}